\documentclass[11pt,reqno,a4paper]{amsart}

\usepackage{graphicx}
\usepackage{color}
\usepackage{transparent}

\usepackage{mathrsfs}
\usepackage{amssymb,amsmath, amsfonts, amsthm}
\usepackage{latexsym}
\usepackage[dvips]{epsfig}

\newtheorem{theorem}{Theorem}[section]
\newtheorem{proposition}[theorem]{Proposition}
\newtheorem{lemma}[theorem]{Lemma}
\newtheorem{corollary}[theorem]{Corollary}
\newtheorem{remark}[theorem]{Remark}



\usepackage{epsfig,color,xcolor}

\newcommand{\ble}{\begin{lemma}}
\newcommand{\ele}{\end{lemma}}
\newcommand{\be}{\begin{equation*}}
\newcommand{\ee}{\end{equation*}}

\newcommand{\bel}{\begin{equation}}
\newcommand{\eel}{\end{equation}}

\newcommand{\ep}{\varepsilon}
\newcommand{\fr}{\frac }

\newcommand{\lap}{\Delta}
\newcommand{\N}{\mathbb{N}}
\newcommand{\na}{\nabla}

\newcommand{\R}{\mathbb{R}}

\renewcommand{\to}{\rightarrow}
\newcommand{\To}{\longrightarrow}

\newcommand{\xip}{x_{i,p}}
\newcommand{\xjp}{x_{j,p}}

\newcommand{\xp}{x_{p}}
\newcommand{\mip}{\mu_{i,p}}

\newcommand{\mup}{\mu_{p}}
\newcommand{\upp}{u_p}

\def\sideremark#1{\ifvmode\leavevmode\fi\vadjust{\vbox to0pt{\vss
 \hbox to 0pt{\hskip\hsize\hskip1em
 \vbox{\hsize2.1cm\tiny\raggedright\pretolerance10000
  \noindent #1\hfill}\hss}\vbox to15pt{\vfil}\vss}}}%

\begin{document}

\numberwithin{equation}{section}
\parindent=0pt
\hfuzz=2pt
\frenchspacing

\title[]{Asymptotic analysis and sign changing bubble towers for Lane-Emden problems}

\author[]{Francesca De Marchis, Isabella Ianni, Filomena Pacella}

\address{Francesca De Marchis, University of Roma {\em Tor Vergata}, Via della Ricerca
Scientifica 1, 00133 Roma, Italy}
\address{Isabella Ianni, Second University of Napoli, V.le Lincoln 5, 81100 Caserta, Italy}
\address{Filomena Pacella, University of Roma {\em Sapienza}, P.le Aldo Moro 8, 00185 Roma, Italy}



\thanks{Research partially supported by FIRB project {\sl Analysis and Beyond} and PRIN 2009-WRJ3W7 grant. }

\begin{abstract} We consider the semilinear Lane-Emden problem
\begin{equation}\label{problemAbstract}\left\{\begin{array}{lr}-\Delta u= |u|^{p-1}u\qquad  \mbox{ in }\Omega\\
u=0\qquad\qquad\qquad\mbox{ on }\partial \Omega
\end{array}\right.\tag{$\mathcal E_p$}
\end{equation}
where $p>1$ and $\Omega$ is a smooth bounded domain of $\R^2$. The aim of the paper is to analyze the asymptotic behavior of  sign changing solutions of \eqref{problemAbstract}, as $p\to+\infty$. 
Among other results we show, under some symmetry assumptions on $\Omega$, that the positive and negative parts of a family of symmetric solutions concentrate at the same point, as $p\to+\infty$, and the limit profile looks like a tower of two bubbles given by a superposition of a regular and a singular solution of the Liouville problem in $\mathbb R^2$.
\end{abstract}

\maketitle

\section{Introduction}
\label{section:intro}
Let $\Omega$ be a smooth bounded domain of $\R^2$.  We consider the Lane-Emden problem
\begin{equation}\label{problem}\left\{\begin{array}{lr}-\Delta u= |u|^{p-1}u\qquad  \mbox{ in }\Omega\\
u=0\qquad\qquad\qquad\mbox{ on }\partial \Omega
\end{array}\right.
\end{equation}
where $p>1$.
\\

The aim of this paper is to contribute to the analysis of the concentration phenomenon for sign changing solutions of problem \eqref{problem} as the exponent $p\rightarrow +\infty$.
\\

In order to explain properly the results and the difficulties related to this investigation let us make a short survey of known results and a comparison with the higher dimensional case when $\Omega\subseteq\mathbb R^N$, $N\geq 3$, $p<\frac{N+2}{N-2}$ and $p\rightarrow \frac{N+2}{N-2}$, i.e. $p$ approaches the critical Sobolev exponent from below.
\\

In higher dimension there is a large literature dealing with the asymptotic behavior of positive solutions, while very little is known for sign changing ones. The reason is that there is a lack of understanding of the finite energy nodal solutions of the ``limit'' problem
\begin{equation}\label{eq:problemaLimiteInHigherDim}
-\Delta Z=|Z|^{\frac{4}{N-2}}Z \ \mbox{ in }\  \mathbb R^N,\ N\geq 3
\end{equation}
which naturally arises in the study of the asymptotic behavior of solutions of \eqref{problem}. We refer to \cite{Struwe_Book} for further details.

The only completely understood case for sign changing solutions, in higher dimension, is when they have low energy, i.e. for solutions $(u_p)$ satisfying
\begin{equation}
\int_{\Omega} |\nabla u_p|^2 dx \rightarrow 2S^{\frac{N}{2}}\ \mbox{ as }\  p \rightarrow \frac{N+2}{N-2},
\end{equation}
where $S$ is the best Sobolev constant.

In \cite{BenAyed_ElMehdi_Pacella} a complete classification of these solutions is provided showing that there are two possibilities. The first one occurs if, for a family of solutions $(u_p)$, there exists a positive constant such that
\begin{equation}\label{eq:condizioneBenAyed_ElMehdi_Pacella}
 \frac{1}{C}\leq \frac{\|u_p^+\|_{\infty}}{\|u_p^-\|_{\infty}}\leq C\ \mbox{ as }\ p\rightarrow \frac{N+2}{N-2}.
\end{equation}
Then $u_p$ blow-up and concentrate at two distinct points of $\Omega$ and suitable scalings of $u_p^+$ and $u_p^-$ (positive and negative part of $u_p$) converge, as $p\rightarrow\frac{N+2}{N-2}$, to a positive regular solution $Z$ of \eqref{eq:problemaLimiteInHigherDim}. In other words the limit profile of $u_p$ is that of two separate bubbles carrying the same energy. Moreover the nodal set touches the boundary of $\Omega$.
The second case arises if \eqref{eq:condizioneBenAyed_ElMehdi_Pacella} does not hold, then it is proved in \cite{BenAyed_ElMehdi_Pacella} that $u_p^+$ and $u_p^-$ blow-up, they concentrate at the same point and they have the local limit profile, after rescaling, of a positive regular solution $Z$ of \eqref{eq:problemaLimiteInHigherDim}. Hence the solution $u_p$ looks like a ``tower'' of two standard bubbles, each one carrying the same energy.
Moreover the nodal set does not touch $\partial\Omega$.
\\

Let us now consider the case when $\Omega\subset\mathbb R^2$.  The first papers where an asymptotic analysis of \eqref{problem}, as $p\rightarrow +\infty$, has been done are \cite{RenWei1, RenWei2}. The authors consider the case of families $(u_p)$ of least energy (hence positive) solutions and in some domains they prove concentration results as well as some asymptotic estimates. Note that the solutions do not blow up as $p\to+\infty$ (unlike the higher dimensional case). Moreover the least energy solutions, for the $2$-dimensional Lane-Emden problem, satisfy the condition
\begin{equation}\label{condizioneLeastEnergy2dim}
p\int_{\Omega}|\nabla u_p|^2dx\rightarrow 8\pi e, \ \mbox{ as }\ p\rightarrow +\infty.
\end{equation}
They left open the question of identifying a limit problem and of detecting the limit of $\|u_p\|_{\infty}$ which was conjectured to be $\sqrt{e}$. Later, inspired by the paper \cite{AdiStruwe} concerning $2$-dimensional problems with critical exponential nonlinearities, Adimurthi and Grossi in \cite{AdiGrossi} showed that a suitable scaling of the least energy solutions $u_p$ converges in $C^1_{loc}(\R^2)$ to a regular solution $U$ of the Liouville problem
\begin{equation}
\label{LiouvilleEquation}
\left\{
\begin{array}{lr}
-\Delta U=e^U\quad\mbox{ in }\R^2\\
\int_{\R^2}e^Udx= 8\pi.
\end{array}
\right.
\end{equation}
They also considered general bounded domains and showed that the $L^{\infty}$-norm of $u_p$ indeed converges to $\sqrt{e}$.\\
Recently in \cite{SantraWei} the authors have analyzed the asymptotic behavior of solutions of some biharmonic equations and pointed out that the same analysis also applies to a family of positive solutions of \eqref{problem} satisfying the condition
\begin{equation}
\label{condition:energiaLimitata}
p\int_{\Omega}|\nabla u_p|^2dx \rightarrow \beta <+\infty \ \mbox{ as }\ p\rightarrow +\infty,
\end{equation}
for some positive constant $\beta\geq 8\pi e$.
Their results show the concentration of the solutions at a finite number of distinct points in $\Omega$, excluding the presence of non simple concentration points (i.e. bubble towers) and give a quantization of the energy.

Concerning sign changing solutions, the asymptotic analysis started in \cite{GrossiGrumiauPacella1} by considering a family $(u_p)$ of low-energy nodal solutions as for the higher dimensional case. Note that, for the $2$-dimensional problem,  this means
\begin{equation}
p\int_{\Omega}|\nabla u_p|^2dx \rightarrow 16\pi e \ \mbox{ as }\ p\rightarrow +\infty,
\end{equation}
which is the analogous of  \eqref{condizioneLeastEnergy2dim} for low-energy positive solutions.

In \cite{GrossiGrumiauPacella1} it is proved that if the minimum and the maximum of $u_p$ are comparable, i.e. if there exists $K\geq 0$ such that
\begin{equation}\label{condizioneComparabilita}
p\left(\|u_p^+\|_{\infty}-\|u_p^-\|_{\infty} \right)\rightarrow K\ \mbox{ as }\ p\rightarrow +\infty
\end{equation}
(which is the analogous of \eqref{eq:condizioneBenAyed_ElMehdi_Pacella} when $N\geq 3$) then $u_p$ concentrate at two distinct points of $\Omega$ and suitable scalings of $u_p^+$ and $u_p^-$ converge to a regular solution $U$ of \eqref{LiouvilleEquation}. Hence the situation is the same as in the higher dimensional case when \eqref{eq:condizioneBenAyed_ElMehdi_Pacella} holds. Moreover in \cite{GrossiGrumiauPacella1} it is proved that if $u_p$ has Morse index $2$ then the maximum and the minimum converge to $\pm\sqrt{e}$ and the nodal line touches the boundary of $\Omega$.

Next, one would like  to consider the case when \eqref{condizioneComparabilita} does not hold and would expect the presence of non-simple concentration points or, in other words, the existence of solutions whose limit profile is given by the superposition of two bubbles, as it happens when $N\geq 3$. The asymptotic analysis in this case looks difficult when $N=2$. However solutions with this limit profile do exist as it has first been proved in \cite{GrossiGrumiauPacella2} by analyzing the asymptotic behavior of least energy radial nodal solutions in the ball. More precisely in \cite{GrossiGrumiauPacella2} the authors prove the following result
\begin{theorem}[Grossi, Grumiau \& Pacella \cite{GrossiGrumiauPacella2}]
Let $(u_p)$ be a family of least energy radial nodal solutions in the unit ball $B$ centered at the origin with $u_p(0)>0$. Then:
\begin{itemize}
\item[i)] a suitable rescaling $z_p^+$ of $u_p^+$ converges in $C^1_{loc}(\R^2)$ to a regular solution $U$ of \eqref{LiouvilleEquation}
\item[ii)] a suitable rescaling and translation $z_p^-$ of $u_p^-$ converges in $C^1_{loc}(\R^2\setminus\{0\})$ to a singular radial solution $V$ of
\begin{equation}
\label{LiouvilleSingularEquation}
\left\{
\begin{array}{lr}
-\Delta V=e^V+ H\delta_0\quad\mbox{ in }\R^2\\
\int_{\R^2}e^Vdx<\infty
\end{array}
\right.
\end{equation}
where $H$ is a negative suitable constant and $\delta_0$ is the Dirac measure centered at $0$.
\end{itemize}
Moreover
\begin{eqnarray*}
&&\|u_p^+\|_{\infty}\rightarrow \alpha^+\simeq 2.46 >\sqrt{e}\\
&&\|u_p^-\|_{\infty}\rightarrow \alpha^-\simeq 1.17 <\sqrt{e}\\
&&p\int_B |\nabla u_p|^2 dx\rightarrow C\simeq 332 >16\pi e\\
&& pu_p(x)\rightarrow 2\pi\gamma G(x,0)=\gamma \log |x| \ \mbox{ as }\ p\rightarrow +\infty
\end{eqnarray*}
for some $\gamma>0$, where $G$ is the Green function on the ball.
\end{theorem}
This result shows a substantial difference between the cases $N=2$ and $N\geq 3$. It means that for $N=2$ there exist solutions which asymptotically look like the superposition of different bubbles given by a regular solution of \eqref{LiouvilleEquation} and a singular solution of \eqref{LiouvilleSingularEquation}. Moreover each bubble does not carry the same energy (unlike when $N\geq 3$).\\

One of the main results of this paper shows that the same phenomenon appears in other domains, different from balls, under some symmetry assumption.\\

We do this through the asymptotic analysis of a family of sign changing solutions found recently in \cite{DeMarchisIanniPacella}.
The main feature of these solutions is that they have an interior nodal line which does not intersect the fixed point of the symmetry group of the domain. We believe that this is the peculiarity of solutions having a bubble tower profile.\\

To state precisely our result we introduce some notations. For a given family $(u_p)$ of sign changing solutions of \eqref{problem} we denote by

\begin{itemize}
\item $NL_p$ the nodal line of $\upp$
\item $x_p^{\pm}$ the maximum/minimum point in $\Omega$ of $u_p$, i.e. $u_p(x_p^{\pm})=\pm\|u_p^{\pm}\|_{\infty}$
\item $\mu_p^{\pm}:=\frac{1}{\sqrt{p|u_p(x_p^{\pm})|^{p-1}}}$
\item $\widetilde{\Omega}_p^{\pm}:=\frac{\Omega-x_p^{\pm}}{\mu_p^{\pm}}=\{x\in\mathbb R^2: x_p^{\pm}+\mu_p^{\pm}x\in \Omega\}$.
\end{itemize}

\

We prove
\begin{theorem}
\label{TeoremaPrincipaleCasoSimmetrico}
Let $\Omega\subset\R^2$ be a simply connected bounded smooth domain, invariant under the action of a cyclic group $G$ of
rotations about the origin (hence the origin $O\in\Omega$),
with $|G|\geq 4e$ ($|G|$ is the order of $G$).
Let $(\upp)$ be a family of sign changing $G$-symmetric  solutions of \eqref{problem} with two nodal regions, $NL_p\cap\partial\Omega=\emptyset$, $O\not\in NL_p$ and
\bel\label{assumptionEnergyINIZIO}
p\int_{\Omega}|\nabla\upp|^2 dx\leq\alpha\,8\pi e
\eel
for some  $\alpha<5$ and $p$ large. Then, assuming w.l.o.g. that $\|u_p\|_{\infty}=\|u_p^+\|_{\infty}$, we have
\begin{itemize}
\item[i)] $|x_p^{\pm}|\rightarrow O \ \mbox{ as }\ p\rightarrow +\infty$;
\item[ii)] $NL_p$ shrinks to the origin as $p\rightarrow +\infty$;
\item[iii)] the rescaled functions $v_p^+(x):=p\frac{u_p(x_p^++\mu_p^+ x)-u_p(x_p^+)}{u_p(x_p^+)}$, about the maximum point, defined in $\widetilde{\Omega}_p^{+}$ converge (up to a subsequence) to the regular solution $U$ of \eqref{LiouvilleEquation} with $U(0)=0$, in $C^1_{loc}(\R^2)$;
\item[iv)] the rescaled functions $v_p^-(x):=p\frac{u_p(x_p^-+\mu_p^- x)-u_p(x_p^-)}{u_p(x_p^-)}$, about the minimum point, defined in $\widetilde{\Omega}_p^{-}$ converge (up to a subsequence) to a singular solution  of \eqref{LiouvilleSingularEquation} for some suitable constant $H$, in
$C^1_{loc}(\R^2\setminus\{x_{\infty}\})$, where $x_{\infty}$ is a point in $\R^2$ with $|x_{\infty}|=\ell=\lim_{p\rightarrow +\infty}\frac{|x_p^-|}{\mu_p^-}>0$;
\item[v)] $\sqrt{p}u_p\rightarrow 0$ in $C^1_{loc}(\bar\Omega\setminus\{0\})$ as $p\rightarrow +\infty$.
\end{itemize}
\end{theorem}

\

\begin{remark} The existence of sign changing solutions satisfying the hypotheses of Theorem \ref{TeoremaPrincipaleCasoSimmetrico} has been proved in \cite{DeMarchisIanniPacella} for any simply connected $G$-symmetric smooth bounded domain with $|G|\geq 4$.
\end{remark}

The results of Theorem \ref{TeoremaPrincipaleCasoSimmetrico} show that both $u^+$ and $u^-$ concentrate at the origin, and, after the above rescalings, they have the limit profile of a regular and a singular solution of the Liouville equation in $\R^2$.
\\

As far as we know, this is the first result of this kind for problem \eqref{problem} in domains different from the ball.\\

The starting point to prove Theorem \ref{TeoremaPrincipaleCasoSimmetrico} is an asymptotic analysis of a general family $(u_p)$ of sign changing solutions of
\eqref{problem} satisfying the condition \eqref{condition:energiaLimitata}. These first results, inspired by the paper \cite{Druet} (see also \cite{DruetHebeyRobert}) can be viewed as a first step towards the analysis of the asymptotic behavior of general sign changing solutions of \eqref{problem}. This kind of profile decomposition results have been proved for several other kind of problems and go back to the papers of Brezis-Coron (\cite{Brezis, BC1, BC2}) whose proofs apply also to critical exponent problems (see for instance \cite{GrossiPacella}).\\
		
Next we use the symmetry assumptions to prove that the maximum points $x_p^+$ converge to the origin as well as any other concentration points.\\

The hardest part of the asymptotic analysis is to prove that the rescaling about the minimum point $x_p^-$ converges to a radial singular solution of a singular Liouville problem in $\R^2$. Indeed, while the rescaling of $u_p$ about the maximum point $x_p^+$ can be studied in a ``canonical'' way, the analysis of the rescaling about $x_p^-$ requires additional arguments. In particular the presence of the nodal line, with an unknown geometry, gives difficulties which, obviously, are not present when dealing with positive solutions or with radial sign changing solutions. Also the proofs of the results for nodal radial solutions of \cite{GrossiGrumiauPacella2} cannot be of any help since they depend strongly on $1$-dimensional estimates.\\

We hope that our results can help to understand better sign changing solutions of other $2$-dimensional nonlinear problems as in \cite{GrossiPistoia}, whose result for $\sinh$-Poisson
equations was inspired by \cite{GrossiGrumiauPacella2}.\\

Finally we would like to point out that these bubble-tower solutions of \eqref{problem} are also interesting in the study of the associated heat flow because they induce a peculiar blow-up phenomenon (see \cite{CazenaveDicksteinWeissler, DicksteinPacellaSciunzi, MarinoPacellaSciunzi}).
\\

We delay to Section \ref{sectionOpenProblems} some further comments and open problems.\\

The outline of the paper is the following. In Section \ref{SectionGeneralAnalysis} we show some results about the asymptotic analysis of sign-changing solutions of \eqref{problem} in general, not necessarily symmetric, domains. In Section \ref{Section:GSymmetricMax} we study the behavior of the solutions around the maximum points, while Section \ref{Section:GSymmetricMin} is devoted to analyze the rescaling about the minimum points.  Finally in Section \ref{sectionOpenProblems} we prove some further properties of the solutions and discuss some open questions.
\

\

\section{Asymptotic analysis in general domains} \label{SectionGeneralAnalysis}

In this section we study the asymptotic behavior of a family $(u_p)_{p>1}$ of sign changing solutions of \eqref{problem} satisfying the energy condition
\begin{equation}
\label{energylimit}
p\int_{\Omega}|\nabla u_p|^2 dx\rightarrow\beta,\ \mbox{ for some } \beta\in\R,\ \mbox{ as }\ p\rightarrow +\infty.
\end{equation}
We follow the approach of \cite{Druet} where positive solutions of semilinear elliptic problems with critical exponential nonlinearities in $2$-dimension were studied.\\

We denote by $E_p$ the energy functional associated to \eqref{problem}, i.e.
\[
E_p(u):=\frac{1}{2}\|\nabla u\|^2_{2}-\frac{1}{p+1}\|u\|_{p+1}^{p+1},\ \ u\in H^1_0(\Omega)
\]
and recall that for a solution $u$  of \eqref{problem}
\begin{equation}
\label{energiaSuSoluzioni}
E_p(u)=(\frac12-\frac1{p+1})\|\nabla u\|^2_2=(\frac12-\frac1{p+1})\|u\|^{p+1}_{p+1}.
\end{equation}

Given a family $(u_p)$ of solutions of \eqref{problem}
and assuming that there exists  $n\in\N\setminus\{0\}$ families of points $(\xip)$, $i=1,\ldots,n$  in $\Omega$ such that
\begin{equation}
\label{muVaAZero}
p|\upp(\xip)|^{p-1}\to+\infty\ \mbox{ as }\ p\to+\infty,
\end{equation}
we define the parameters $\mip$ by
\bel\label{mip}
\mip^{-2}=p |\upp(\xip)|^{p-1},\ \mbox{ for all }\ i=1,\ldots,n.
\eel
By \eqref{muVaAZero} it is clear that $\mip\to0$ as $p\to+\infty$ and that
\begin{equation}\label{RemarkMaxCirca1}
 \forall \epsilon>0 \;\: \exists\, p_{i,\epsilon}\ \mbox{ such that }\
 |u_p(\xip)|\geq 1-\epsilon,\ \ \forall p\geq p_{i,\epsilon}.
\end{equation}
Then we define the concentration set
\bel\label{S}
\mathcal{S}=\left\{\lim_{p\to+\infty}\xip,\,i=1,\ldots,n\right\}\subset\bar\Omega
\eel
and the function
\bel\label{RNp}
R_{n,p}(x)=\min_{i=1,\ldots,n} |x-\xip|, \ \forall x\in\Omega.
\eel

Finally we introduce the following properties:
\begin{itemize}
\item[$(\mathcal{P}_1^n)$] For any $i,j\in\{1,\ldots,n\}$, $i\neq j$,
\[
\lim_{p\to+\infty}\fr{|\xip-\xjp|}{\mip}=+\infty.
\]
\item[$(\mathcal{P}_2^n)$] For any $i=1,\ldots,n$,
\[
v_{i,p}(x):=\fr{p}{\upp(\xip)}(\upp(\xip+\mip x)-\upp(\xip))\To U(x)
\]
in $C^1_{loc}(\R^2)$ as $p\to+\infty$, where
\begin{equation}\label{v0}
U(x)=\log\left(\fr1{1+\fr18 |x|^2}\right)^2
\end{equation}
is the solution of $-\lap U=e^{U}$ in $\R^2$, $U\leq 0$, $U(0)=0$ and $\int_{\mathbb{R}^2}e^{U}=8\pi$.
\item[$(\mathcal{P}_3^n)$] There exists $C>0$ such that
\[
p R_{n,p}(x)^2 |\upp(x)|^{p-1}\leq C
\]
for all $p$ sufficiently large and all $x\in \Omega$.
\end{itemize}

\

\ble\label{lemma:BoundEnergia}
Let $(\upp)$ be a family of solutions to \eqref{problem} satisfying \eqref{energylimit}. Then
\begin{itemize}
\item[\emph{$(i)$}] $\exists C>0$ such that $\|\upp\|_{L^\infty(\Omega)}\leq C,$ for all $p>1$.

\item[\emph{$(ii)$}] If $\upp$ changes sign, then  $\|u_p^\pm\|_{L^\infty(\Omega)}^{p-1}\geq \lambda_1$ where $\lambda_1:=\lambda_1(\Omega)$ is the first eigenvalue of the operator $-\Delta$ in $H^1_0(\Omega)$. In particular for the points $x_p^{\pm}$, where the maximum and the minimum are achieved, the analogous of \eqref{muVaAZero} and \eqref{RemarkMaxCirca1} hold.

\item[\emph{$(iii)$}] If, for $n\in\N\setminus\{0\}$, the properties $(\mathcal{P}_1^n)$ and $(\mathcal{P}_2^n)$ hold for families $(\xip)_{i=1,\ldots,n}$ of points satisfying \eqref{muVaAZero}, then
\[
p\int_\Omega |\na\upp|^2\,dx\geq8\pi\sum_{i=1}^n \alpha_i^2+o_p(1)\ \mbox{ as }p\rightarrow +\infty,
\]
where  $\alpha_i:=\liminf_{p\to+\infty}|\upp(\xip)|$.

\item[\emph{$(iv)$}] $\sqrt{p}u_p\rightharpoonup 0$ in $H^1_{0}(\Omega)$ as $p\rightarrow +\infty$.
\end{itemize}
\ele

\begin{proof} The proof of
\emph{$(i)$} is the same as in Proposition 2.7 of \cite{GrossiGrumiauPacella1}, while the proof of
\emph{$(ii)$} is similar to that of Proposition 2.5 of \cite{GrossiGrumiauPacella1}. To prove \emph{$(iii)$} let us write, for any $R>0$
\begin{eqnarray}\label{appo1}
p\int_{B_{R\mu_{i,p}}(\xip)} |\upp|^{p+1}\,dx&=&\int_{B_R(0)}\fr{|\upp(\xip+\mip y)|^{p+1}}{|\upp(\xip)|^{p-1}}\,dy\nonumber\\
 &=&\int_{B_R(0)}\left|1+\fr{v_{i,p}(y)}{p}\right|^{p+1}\upp^2(\xip)\,dy\ \ \ \
\end{eqnarray}
where $B_{\rho}(a)$ denotes the ball of center $a$ and radius $\rho$.
Thanks to $(\mathcal{P}_2^n)$, we have
\bel\label{appo2}
\|v_{i,p}-U\|_{L^{\infty}(B_R(0))}=o_p(1)\ \mbox{ as }p\rightarrow +\infty.
\eel
Thus by \eqref{appo1}, \eqref{appo2} and Fatou's lemma
\bel\label{appo3}
\liminf_{p\to+\infty}\,\left(p\int_{B_{R\mu_{i,p}}(\xip)}|\upp|^{p+1}\,dx\right)\geq\alpha^2_i\int_{B_R(0)} e^{U}\, dx.
\eel

Moreover by virtue of $(\mathcal{P}_1^n)$ it is not hard to see that $B_{R\mip}(\xip)\cap B_{R\mu_{j,p}}(x_{j,p})=\emptyset$ for all $i\neq j$. Hence, in particular, thanks to \eqref{appo3}
\[
\liminf_{p\to+\infty}\,\left(p\int_\Omega |\upp|^{p+1}\,dx\right)\geq \sum_{i=1}^n \left(\alpha^2_i\int_{B_R(0)} e^{U}\, dx\right).
\]
At last, since this holds for any $R>0$, we get
\[
p\int_\Omega|\na\upp|^2\,dx=p\int_\Omega |\upp|^{p+1}\,dx\geq \sum_{i=1}^n \alpha_i^2 \int_{\R^2}e^{U}\,dx
+o(1)=8\pi\sum_{i=1}^n \alpha_i^2+o(1) \ \mbox{ as }p\rightarrow +\infty.\]

To prove \emph{$(iv)$} let us note that, since \eqref{energylimit} holds,
 there exists $w\in H^1_0(\Omega)$ such that, up to a subsequence,  $\sqrt{p}u_p \rightharpoonup  w$ in $H^1_0(\Omega)$. We want to show that $w=0$ a.e. in $\Omega$.

Using the equation \eqref{problem}, for any test function $\varphi\in C^{\infty}_0(\Omega)$, we have
\begin{eqnarray*}
\int_{\Omega}\nabla (\sqrt{p}u_p)\nabla\varphi \,dx &=& \sqrt{p}\int_{\Omega}|u_p|^{p-1}u_p\varphi\, dx
\\
&\leq &\frac{\|\varphi\|_{\infty}}{\sqrt{p}} p \int_{\Omega}|u_p|^{p}\,dx
\\
&\leq &\frac{\|\varphi\|_{\infty}}{\sqrt{p}} C
\end{eqnarray*}
for $p$ large since, by \eqref{energylimit} and \eqref{energiaSuSoluzioni}, $\int_{\Omega}|u_p|^{p}\,dx \leq \left(\int_{\Omega}|u_p|^{p+1}\,dx\right)^{\frac{p}{p+1}}|\Omega|^{\frac{1}{p+1}}\leq\frac{C}{p}$. Hence
\[
\int_{\Omega}\nabla w\nabla\varphi\,dx=0\quad \forall\varphi\in C^{\infty}_0(\Omega),
\]
which implies that
 $w=0$ a.e. in $\Omega$.
\end{proof}

\

Next proposition gives the main result of this section.

\begin{proposition}\label{prop:x1N}
Let $(\upp)$ be a family of solutions to \eqref{problem} and assume that \eqref{energylimit} holds. Then there exist $k\in\N\setminus\{0\}$ and $k$ families of points $(\xip)$ in $\Omega$  $i=1,\ldots, k$ such that, after passing to a sequence, $(\mathcal{P}_1^k)$, $(\mathcal{P}_2^k)$, and $(\mathcal{P}_3^k)$ hold. Moreover, given any family of points $x_{k+1,p}$, it is impossible to extract a new sequence from the previous one such that $(\mathcal{P}_1^{k+1})$, $(\mathcal{P}_2^{k+1})$, and $(\mathcal{P}_3^{k+1})$ hold with the sequences $(\xip)$, $i=1,\ldots,k+1$. At last, we have
\begin{equation}\label{pu_va_a_zero}
\sqrt{p}\upp\to 0\quad\textrm{ in $C^1_{loc}(\bar\Omega\setminus\mathcal{S})$ as $p\to+\infty$.}
\end{equation}
\end{proposition}
\begin{proof}
We mainly follow the proof of Proposition 1 of \cite{Druet}, but we have to deal with an extra-difficulty because we allow the solutions $\upp$ to be sign-changing. We divide the proof in several steps and all the claims are up to a subsequence.

\

\textit{\underline{Step 1.}\\ There exists a family $(\xip)$ of points in $\Omega$ such that, after passing to a sequence $(\mathcal{P}^1_2)$ holds.}

\

\textit{\underline{Proof of Step 1}}\\
We let $x_{1,p}$ be a point in $\Omega$ where $|\upp|$ achieves its maximum. Without loss of generality we can assume that
\bel\label{x1p}
\upp(x_{1,p})=\max_\Omega \upp>0.
\eel
By Lemma \ref{lemma:BoundEnergia} \emph{(ii)} we have
$$
p\upp(x_{1,p})^{p-1}\to+\infty\quad\textrm{as $p\to+\infty$},
$$
so that, defining (as in \eqref{mip})
\[
\mu_{1,p}^{-2}=p\upp^{p-1}(x_{1,p}),
\]
we have that $\mu_{1,p}\to 0$. Let
\[
\widetilde{\Omega}_{1,p}=\frac{\Omega - x_{1,p}}{\mu_{1,p}}=\{x\in\R^2\,:\,x_{1,p}+\mu_{1,p}x\in\Omega\}
\]
and for $x\in\widetilde{\Omega}_{1,p}$
\[
v_{1,p}(x)=\fr{p}{\upp(x_{1,p})}(\upp(x_{1,p}+\mu_{1,p}x)-\upp(x_{1,p})).
\]
By \eqref{x1p}, we have
\begin{equation}\label{b}
v_{1,p}(0)=0\quad\textrm{and}\quad v_{1,p}\leq 0\textrm{  in $\widetilde{\Omega}_{1,p}$}
\end{equation}
moreover $v_{1,p}$ solves
\bel\label{lapv1p}
-\lap v_{1,p}=\left|1+\frac{v_{1,p}}{p}\right|^p\left(1+\frac{v_{1,p}}{p}\right)\quad\textrm{in $\widetilde{\Omega}_{1,p}$},
\eel
with \[\left|1+\frac{v_{1,p}}{p}\right|\leq 1\]
and \[v_{1,p}=-p\quad\textrm{on $\partial\widetilde{\Omega}_{1,p}$.}\]
Then
\bel\label{boundlap1}
|-\lap v_{1,p}|\leq 1\quad\textrm{in $\widetilde{\Omega}_{1,p}$}.
\eel
Using  \eqref{b} and \eqref{boundlap1} we prove that
\begin{equation}\label{tuttoR2}
\widetilde{\Omega}_{1,p}\to\R^2\ \mbox{ as }p\to+\infty.
\end{equation}Indeed since $\mu_{1,p}\rightarrow 0$ as $p\rightarrow + \infty$, either $\widetilde{\Omega}_{1,p}\rightarrow\mathbb R^2$ or $\widetilde{\Omega}_{1,p}\rightarrow\mathbb R\times ]-\infty, R[$ as $p\rightarrow +\infty$ for some $R\geq 0$ (up to a rotation). In the second case we let
\[
v_{1,p}=\varphi_p+\psi_p \ \mbox{ in }\widetilde{\Omega}_{1,p}\cap B_{2R+1}(0)
\]
with
$-\lap\varphi_p=-\lap v_{1,p}$ in $\widetilde{\Omega}_{1,p}\cap B_{2R+1}(0)$ and $\psi_p= v_{1,p}$ in $\partial\left(\widetilde{\Omega}_{1,p}\cap B_{2R+1}(0)\right)$.

Thanks to \eqref{boundlap1}  we have, by standard elliptic theory, that $\varphi_p$ is uniformly bounded in $\widetilde{\Omega}_{1,p}\cap B_{2R+1}(0)$. The function $\psi_p$ is harmonic in $\widetilde{\Omega}_{1,p}\cap B_{2R+1}(0)$, bounded from above by \eqref{b} and satisfies  $\psi_p=-p\rightarrow -\infty$ on $\partial\widetilde{\Omega}_{1,p}\cap B_{2R+1}(0)$. Since $\partial\widetilde{\Omega}_{1,p}\cap B_{2R+1}(0)\rightarrow (\mathbb{R}\times\{R\})\cap B_{2R+1}(0)$ as $p\rightarrow +\infty$  one easily gets that $\psi_p(0)\rightarrow -\infty$ as $p\rightarrow +\infty$ (if $R=0$ this is trivial, if $R>0$ it follows by Harnack inequality). This is  a contradiction since $\psi_p(0)=-\varphi_p(0)$ and $\varphi_p$ is bounded, hence \eqref{tuttoR2} is proved.
\\

Then for any $R>0$, $B_R(0)\subset\widetilde\Omega_{1,p}$ for $p$ sufficiently large. So $(v_{1,p})$ is a family of nonpositive functions with uniformly bounded Laplacian in $B_R(0)$ and with $v_{1,p}(0)=0$.

Thus, arguing as before, we write $v_{1,p}=\varphi_p+\psi_p $ where $\varphi_p$ is uniformly bounded in $B_R(0)$  and $\psi_p $ is an harmonic function which is uniformly bounded from above. By Harnack inequality, either $\psi_p $ is uniformly bounded in $B_R(0)$ or it tends to $-\infty$ on each compact set of $B_R(0)$.
The second alternative cannot happen because, by definition, $\psi_p(0)=v_{1,p}(0)-\varphi_p(0) =-\varphi_p(0)\geq -C$.
Hence we obtain that $v_{1,p}$ is uniformly bounded in $B_R(0)$, for all $R>0$. After passing to a subsequence, standard elliptic theory gives that $v_{1,p}$ is bounded in $ C^2_{loc}(\mathbb R^2)$ and, on each ball, $1+\frac{v_{1,p}}{p}>0$ for $p$ large. Thus
\bel\label{v1pv0}
v_{1,p}\to U\quad\textrm{in $C^1_{loc}(\R^2)$ as $p\to+\infty$},
\eel
with $U\in C^1(\R^2)$, $U\leq0$ and $U(0)=0$. Thanks to \eqref{lapv1p} and \eqref{v1pv0} we get
that $U$ is a solution of
\[
-\lap U=e^{U}\quad\textrm{in $\R^2$.}
\]
Moreover for any $R>0$, by \eqref{RemarkMaxCirca1}, we have
\begin{eqnarray*}
\int_{B_R(0)}e^{U}dx&\stackrel{\eqref{v1pv0}\,+\,\textrm{{Fatou}}}{\leq} &\int_{B_R(0)}\fr{|\upp(x_{1,p}+\mu_{1,p}x)|^{p+1}}{\upp^{p+1}(x_{1,p})}dx+o_p(1)\\
&=&\fr p{\upp^2(x_{1,p})}\int_{B_{R\mu_{1,p}}(x_{1,p})}|\upp(y)|^{p+1}dy+o_p(1)\\
&\stackrel{\eqref{RemarkMaxCirca1}}{\leq}&\fr p{(1-\ep)^2}\int_{B_{R\mu_{1,p}}(x_{1,p})}|\upp(y)|^{p+1} dy+o_p(1)\\
&\stackrel{\eqref{energylimit}}\leq&\fr p{(1-\ep)^2}\int_{\Omega}|\upp(y)|^{p+1} dy+o_p(1)<C<+\infty,
\end{eqnarray*}
so that $e^{U}\in L^1(\R^2)$. Thus, since $U(0)=0$, by virtue of the classification due to Chen and Li \cite{ChenLi} we obtain
\[
U(x)=\log\left(\fr1{1+\fr18 |x|^2}\right)^2.
\]
Then an easy computation shows that $\int_{\mathbb{R}^2}e^{U}=8\pi$.  This ends the proof of \textit{Step 1}.

\

\

\textit{\underline{Step 2.}\\ Assume that $(\mathcal{P}_1^n)$ and $(\mathcal{P}_2^n)$ hold for some $n\in\N\setminus\{0\}$. Then either $(\mathcal{P}_1^{n+1})$ and $(\mathcal{P}_2^{n+1})$ hold or $(\mathcal{P}_3^n)$ holds, namely there exists $C>0$ such that
$$
p R_{n,p}(x)^2 |\upp(x)|^{p-1}\leq C
$$
for all $p$ sufficiently large and all $x\in\Omega$, with $R_{n,p}$  defined as in \eqref{RNp}.
}

\

\textit{\underline{Proof of Step 2}}\\
Let $n\in\N\setminus\{0\}$ and assume that $(\mathcal{P}_1^n)$ and $(\mathcal{P}_2^n)$ hold while
\bel\label{Pknonvale}
\sup_{x\in\Omega}\left(p R_{n,p}(x)^2 |\upp(x)|^{p-1}\right)\to+\infty\quad\textrm{as $p\to+\infty$}.
\eel
We must prove that $(\mathcal{P}_1^{n+1})$ and $(\mathcal{P}_2^{n+1})$ hold. We let $x_{n+1,p}\in\bar\Omega$ be such that
\bel\label{xk+1}
p R_{n,p}(x_{n+1,p})^2 |\upp(x_{n+1,p})|^{p-1}=\sup_{x\in\Omega}\left(p R_{n,p}(x)^2 |\upp(x)|^{p-1}\right).
\eel
Clearly $x_{n+1,p}\in\Omega$ because $\upp=0$ on $\partial \Omega$.
By \eqref{xk+1} and since $\Omega$ is bounded it is clear that
$$
p|\upp(x_{n+1,p})|^{p-1}\to+\infty\quad\textrm{as $p\to+\infty.$}
$$
We claim that
\bel\label{claimxk+1}
\fr{|\xip-x_{n+1,p}|}{\mu_{i,p}}\to+\infty\quad\textrm{as $p\to+\infty$}
\eel
for all $i=1,\ldots,n$ and $\mu_{i,p}$ as in \eqref{mip}. In fact, assuming by contradiction that there exists $i\in\{1,\ldots,n\}$ such that $|\xip-x_{n+1,p}|/\mip\to R$ as $p\to+\infty$ for some $R\geq0$, thanks to $(\mathcal{P}_2^n)$, we get
\[
\lim_{p\to+\infty}p|\xip-x_{n+1,p}|^2 |\upp(x_{n+1,p})|^{p-1}= R^2\left(\fr{1}{1+\fr18R^2}\right)^2<+\infty,
\]
against \eqref{xk+1}. \\
Setting
\bel\label{mk+1}
(\mu_{n+1,p})^{-2}=p |\upp(x_{n+1,p})|^{p-1}.
\eel
by \eqref{Pknonvale} and \eqref{xk+1} we deduce that
\bel\label{Rkmk+1}
\fr{R_{n,p}(x_{n+1,p})}{\mu_{n+1,p}}\to+\infty\quad\textrm{as $p\to+\infty$.}
\eel
Then \eqref{mk+1}, \eqref{Rkmk+1} and $(\mathcal{P}_1^n)$ imply that $(\mathcal{P}_1^{n+1})$ holds with the added sequence $(x_{n+1,p})$. Indeed we define the rescaled domain
\[
\widetilde{\Omega}_{n+1,p}=\{x\in\R^2\,:\, x_{n+1,p}+\mu_{n+1,p}x\in\Omega\},
\]
and, for $x\in\widetilde{\Omega}_{n+1,p}$, the rescaled functions
\bel\label{vk+1p}
v_{n+1,p}(x)=\fr{p}{\upp(x_{n+1,p})}(\upp(x_{n+1,p}+\mu_{n+1,p}x)-\upp(x_{n+1,p})),
\eel
which, by \eqref{problem}, satisfy
\bel\label{lapvk+1p1}
-\lap v_{n+1,p}(x)=\fr{|\upp(x_{n+1,p}+\mu_{n+1,p}x)|^{p-1}\upp(x_{n+1,p}+\mu_{n+1,p}x)}{|\upp(x_{n+1,p})|^{p-1}\upp(x_{n+1,p})}\qquad\textrm{in $\widetilde{\Omega}_{n+1,p}$},
\eel
or equivalently
\bel\label{lapvk+1p2}
-\lap v_{n+1,p}(x)=\left|1+\fr{v_{n+1,p}(x)}{p}\right|^{p-1}\left(1+\fr{v_{n+1,p}(x)}{p}\right)\qquad\textrm{in $\widetilde{\Omega}_{n+1,p}$}.
\eel
Fix $R>0$ and let $(z_{p})$ be any point in $\widetilde{\Omega}_{n+1,p}\cap B_R(0)$, whose corresponding points in $\Omega$ is
$$
x_p=x_{n+1,p}+\mu_{n+1,p} z_p.
$$
Thanks to the definition of $x_{n+1,p}$ we have
\begin{equation}\label{r}
p R_{n,p}(x_p)^2|\upp(x_p)|^{p-1}\leq p R_{n,p}(x_{n+1,p})^2|\upp(x_{n+1,p})|^{p-1}.
\end{equation}
Since $|x_p-x_{n+1,p}|\leq R\mu_{n+1,p}$ we have
\begin{eqnarray*}
R_{n,p}(x_p)&\geq&\min_{i=1,\ldots,n}|x_{n+1,p}-\xip|-|x_p-x_{n+1,p}|\\
&\geq&R_{n,p}(x_{n+1,p})-R\mu_{n+1,p};
\end{eqnarray*}
and, analogously,
$$
R_{n,p}(x_p)\leq R_{n,p}(x_{n+1,p})+R\mu_{n+1,p}.
$$
Thus, by \eqref{Rkmk+1} we get
$$
R_{n,p}(x_p)=(1+o(1))R_{n,p}(x_{n+1,p})
$$
and in turn from \eqref{r}
\bel\label{1+o1}
|\upp(x_p)|^{p-1}\leq(1+o(1))|\upp(x_{n+1,p})|^{p-1}.
\eel

In the following we show that for any compact subset $K$ of $\R^2$
\begin{equation}\label{boundLaplaciano}
-1+o(1)\leq-\lap v_{n+1,p}\leq 1+o(1)
\qquad\textrm{in $\widetilde{\Omega}_{n+1,p}\cap K$}
\end{equation}
and
\begin{equation}\label{vNeg}
\limsup_{p\to+\infty}\sup_{\widetilde \Omega_{n+1,p}\cap K}v_{n+1,p}\leq0,
\end{equation}.
In order to prove \eqref{boundLaplaciano} and \eqref{vNeg} we will distinguish different cases.\\

\begin{itemize}
\item[$(i)$] \emph{Assume that $v_{n+1,p}(z_p)\geq0$.}\\
If $\upp(x_{n+1,p})>0$ then $\upp(x_p)=\fr{\upp(x_{n+1,p})}{p}v_{n+1,p}(z_p)+\upp(x_{n+1,p})\geq\upp(x_{n+1,p})>0,
$
while if $\upp(x_{n+1,p})<0$ then
$
\upp(x_p)=\fr{\upp(x_{n+1,p})}{p}v_{n+1,p}(z_p)+\upp(x_{n+1,p})\leq\upp(x_{n+1,p})<0.
$
So in both cases
\[
\frac{\upp(x_p)}{\upp(x_{n+1,p})}=\frac{|\upp(x_p)|}{|\upp(x_{n+1,p})|}
\]
By \eqref{1+o1} we get $|\upp(x_p)|^p\leq(1+o(1))|\upp(x_{n+1,p})|^p$ and so by \eqref{lapvk+1p1}
\begin{equation}\label{firstHalfLapl}
(0\leq)-\lap v_{n+1,p}(z_p)=\frac{|\upp(x_p)|^p}{|\upp(x_{n+1,p})|^p}
\leq 1+o(1).
\end{equation}
Moreover, since \eqref{lapvk+1p2} implies $-\lap v_{n+1,p}(z_p)=e^{v_{n+1,p}(z_p)}+o(1)$, we get
\[
\limsup_{p\to+\infty}v_{n+1,p}(z_p)\leq 0.
\]
By the arbitrariness of $z_p$ we obtain
\eqref{vNeg}.\\
\item[$(ii)$] \emph{Assume that $v_{n+1,p}(z_p)\leq0$.} We distinguish two cases.
\begin{itemize}
\item[$1.$]{\emph{$\upp(x_{n+1,p})>0$.}}\\
Then
$
\upp(x_p)=\fr{\upp(x_{n+1,p})}p v_{n+1,p}(z_p)+\upp(x_{n+1,p})\leq \upp(x_{n+1,p}).
$
So either \begin{itemize}
\item[$1.1$]$\upp(x_p)\geq0$ and then $(0\leq) -\lap v_{n+1,p}(z_p)\leq1$, or
\item[$1.2$]$\upp(x_p)<0$ and then by \eqref{1+o1} $$0\geq-\lap v_{n+1,p}(z_p)=-\fr{|\upp(x_p)|^{p}}{\upp(x_{n+1,p})^p}\geq-1+o(1).$$
\end{itemize}
\item[$2.$]{\emph{$\upp(x_{n+1,p})<0$.}}\\ Then
$
\upp(x_p)=\fr{\upp(x_{n+1,p})}p v_{n+1,p}(z_p)+\upp(x_{n+1,p})\geq \upp(x_{n+1,p}).
$
So either \begin{itemize}
\item[$2.1$]$\upp(x_p)\leq0$ and then $(0\leq) -\lap v_{n+1,p}(z_p)\leq1$, or
\item[$2.2$]$\upp(x_p)>0$ and then by \eqref{1+o1} $$0\geq-\lap v_{n+1,p}(z_p)=-\fr{\upp(x_p)^{p}}{|\upp(x_{n+1,p}|)^p}\geq-1+o(1).$$
\end{itemize}
\end{itemize}
In the end in both case $1$ and case $2$ we have proved that, as $p\rightarrow +\infty$
\begin{equation}\label{firstHalfLap2}
-1+o(1)\leq-\lap v_{n+1,p}(z_p)\leq 1+o(1)
\end{equation}
\end{itemize}
Putting together \eqref{firstHalfLapl} and \eqref{firstHalfLap2} it follows that \eqref{boundLaplaciano} holds.
\\
\\
\\

Using \eqref{boundLaplaciano} and \eqref{vNeg} we can prove, similarly as in \textit{Step 1}, that
\begin{equation}\widetilde{\Omega}_{n+1,p}\rightarrow\mathbb R^2\ \mbox{ as }p\rightarrow +\infty.
\end{equation}

Indeed, since $\mu_{n+1,p}\rightarrow 0$ as $p\rightarrow + \infty$, either $\widetilde{\Omega}_{n+1,p}\rightarrow\mathbb R^2$ or $\widetilde{\Omega}_{n+1,p}\rightarrow\mathbb R\times ]-\infty, R[$ as $p\rightarrow +\infty$ for some $R\geq 0$ (up to a rotation). In the second case we let
\[
v_{n+1,p}=\varphi_p+\psi_p \ \mbox{ in }\widetilde{\Omega}_{n+1,p}\cap B_{2R+1}(0)
\]
with
$-\lap\varphi_p=-\lap v_{n+1,p}$ in $\widetilde{\Omega}_{n+1,p}\cap B_{2R+1}(0)$ and $\psi_p= v_{n+1,p}$ in $\partial\left(\widetilde{\Omega}_{n+1,p}\cap B_{2R+1}(0)\right)$.

Thanks to \eqref{boundLaplaciano},  since $\varphi_p= v_{n+1,p}$ in $\partial\left(\widetilde{\Omega}_{n+1,p}\cap B_{2R+1}(0)\right)$, we have, by standard elliptic theory, that $\varphi_p$ is uniformly bounded in $\widetilde{\Omega}_{n+1,p}\cap B_{2R+1}(0)$. The function $\psi_p$ is  harmonic  in $\widetilde{\Omega}_{n+1,p}\cap B_{2R+1}(0)$, bounded from above by \eqref{vNeg} and satisfies  $\psi_p=-p\rightarrow -\infty$ on $\partial\widetilde{\Omega}_{n+1,p}\cap B_{2R+1}(0)$. Since $\partial\widetilde{\Omega}_{n+1,p}\cap B_{2R+1}(0)\rightarrow (\mathbb{R}\times\{R\})\cap B_{2R+1}(0)$ as $p\rightarrow +\infty$   one easily gets that $\psi_p(0)\rightarrow -\infty$ as $p\rightarrow +\infty$ (if $R=0$ this is trivial, if $R>0$ it follows by Harnack inequality). This is  a contradiction since $\psi_p(0)=-\varphi_p(0)$ and $\varphi_p$ is bounded. Therefore the limit domain of $\widetilde{\Omega}_{n+1,p}$ is the whole $\R^2$.
\\
\\
\\
Then for any $R>0$, $B_R(0)\subset\widetilde{\Omega}_{n+1,p}$ for $p$ large enough and $v_{n+1,p}$ are  functions with uniformly bounded laplacian in $B_R(0)$ and with $v_{n+1,p}(0)=0$. So, by Harnack inequality, $v_{n+1,p}$ is uniformly bounded in $B_R(0)$ for all $R>0$ and then $v_{n+1,p}\to U$ in $C^1_{loc}(\R^2)$ as $p\to+\infty$ with $U\in C^1(\R^2)$,  $U(0)=0$ and, by \eqref{vNeg}, $U\leq0$.\\
Passing to the limit in \eqref{lapvk+1p2} we get
$$
-\lap v_{n+1,p}(x)\to e^{U(x)}\quad  \textrm{ as $p\to+\infty$},
$$
and so $-\lap U=e^{U}$ in $\R^2$.\\
Next for any $R>0$
$$
\int_{B_R(0)}e^{U}\,dx\leq p\int_{B_{R\mu_{n+1,p}}(x_{n+1,p})}\upp\lap\upp\,dx+o_p(1)\leq p\int_{\Omega}|\na \upp|^2\,dx+o_p(1),
$$
so that $e^{U}\in L^1(\R^2)$. By \cite{ChenLi} and $v_{n+1,p}(0)=0$ we have $U(x)=\log\left(\fr1{1+\fr18 x^2}\right)^2$.

This proves that $(\mathcal{P}_2^{n+1})$ holds with the added points $(x_{n+1,p})$, thus \emph{Step 2} is proved.

\

\

\

\textit{\underline{Step 3.} \\
Proof of Proposition \ref{prop:x1N} completed.}

\

\textit{\underline{Proof of Step 3}}\\
From \emph{Step 1.} we have that $(\mathcal{P}_1^1)$ and $(\mathcal{P}_2^1)$ hold. Then, by \emph{Step 2}, either $(\mathcal{P}_1^2)$ and $(\mathcal{P}_2^2)$ hold or
$(\mathcal{P}_3^1)$ holds. In the last case the assertion is proved with $k=1$. In the first case we go on and proceed with the same alternative until we reach a number $k\in\N\setminus\{0\}$ for which $(\mathcal{P}_1^{k})$, $(\mathcal{P}_2^{k})$ and $(\mathcal{P}_3^{k})$ hold up to a sequence. Note that this is possible because the solutions $u_p$ satisfy  \eqref{energylimit} and Lemma \ref{lemma:BoundEnergia} hols and hence the maximal number $k$ of families of points for which
$(\mathcal{P}_1^{k})$, $(\mathcal{P}_2^{k})$ hold must be finite.

\

Moreover, given any other family of points $x_{k+1,p}$, it is impossible to extract a new sequence from it such that $(\mathcal{P}_1^{k+1})$, $(\mathcal{P}_2^{k+1})$ and $(\mathcal{P}_3^{k+1})$ hold together with the points $(x_{i,p})_{i=1,..,k+1}$. Indeed if $(\mathcal{P}_1^{k+1})$ was verified then $$\frac{|x_{k+1,p}-\xip|}{\mu_{k+1,p}}\to+\infty\qquad\textrm{as $p\to+\infty$,\quad for any $i\in\{1,\ldots,k\}$},$$
but this would contradict $(\mathcal{P}_3^k)$.

\

Finally the proof of \eqref{pu_va_a_zero} is a direct consequence of $(\mathcal{P}^k_3)$. Indeed if $K$ is a compact subset of $\bar\Omega\setminus\mathcal S$ by $(\mathcal{P}^k_3)$ we have that there exists $C_K>0$ such that
$$
p|\upp(x)|^{p-1}\leq C_K\qquad\textrm{for all $x\in K$ and $p$ sufficiently large.}
$$
Then by  \eqref{problem} $\|\lap(\sqrt{p} \upp)\|_{L^\infty(K)}\leq C_K\frac{\|u_p\|_{L^\infty(K)}}{\sqrt{p}}\to0$ uniformly in $K$, as $p\to+\infty$. Hence standard elliptic theory shows that $\sqrt{p}u_p\to w$ in $C^1(K)$, for some $w$. But by  Lemma \ref{lemma:BoundEnergia}-(iv) we know that
$\sqrt{p}u\rightharpoonup 0$, so $w=0$. This ends the proof.
\end{proof}

\

\

We conclude this section showing some consequences of Proposition \ref{prop:x1N}.

\

\begin{remark}\label{rem:nonvedobordo}
Under the assumptions of Proposition \ref{prop:x1N} we have
$$
\fr{dist(\xip,\partial\Omega)}{\mip}\stackrel{p\to+\infty}{\to}+\infty,\qquad\textrm{for all $i\in\{1,\ldots,k\}$}.
$$
\end{remark}
\begin{corollary}\label{cor:nonvedoNL}
Under the assumptions of Proposition \ref{prop:x1N} if $u_p$ is sign-changing it follows that
$$
\fr{dist(\xip,NL_p)}{\mip}\stackrel{p\to+\infty}{\to}+\infty,\qquad\textrm{for all $i\in\{1,\ldots,k\}$}
$$
where, as in the Section \ref{section:intro},  $NL_p$ denotes the  nodal line of $\upp$.\\

As a consequence, for any $i\in\{1,\ldots,k\}$, letting $\mathcal{N}_{i,p}\subset\Omega$ be the nodal domain of $u_p$ containing $x_{i,p}$ and setting $u_p^i:=u_p\chi_{\mathcal{N}_{i,p}}$ ($\chi_A$ is the characteristic function of the set $A$),
then the scaling of $u_p^i$ around $x_{i,p}$:
\[
z_{i,p}(x):=\fr{p}{\upp(\xip)}(\upp^i(\xip+\mip x)-\upp(\xip)),
\]
defined on $\widetilde{\mathcal{N}}_{i,p}:=\frac{\mathcal{N}_{i,p}-x_{i,p}}{\mu_{i,p}}$, converges to $U$ in $C^1_{loc}(\mathbb R^2)$, where $U$ is the same function defined in $(\mathcal{P}_2^k)$.
\end{corollary}
\begin{proof}
Let us suppose by contradiction that
$$
\fr{dist(\xip,NL_p)}{\mip}\stackrel{p\to+\infty}{\to}\ell\geq0,
$$
then there exist $y_p\in NL_p$ such that $\fr{|\xip-y_p|}{\mip}\to\ell$.
Setting
$$
v_{i,p}(x)=\fr{p}{\upp(\xip)}(\upp(\xip+\mip x)-\upp(\xip)),
$$
on the one hand
$$
v_{i,p}(\fr{y_p-\xip}{\mip})=-p\stackrel{p\to+\infty}{\to}-\infty,
$$
on the other hand by $(\mathcal{P}_2^k)$ and up to subsequences
$$
v_{i,p}(\fr{y_p-\xip}{\mip})\stackrel{p\to+\infty}{\to}U(x_{\infty})>-\infty,
$$
where $x_\infty=\lim_{p\rightarrow +\infty} \fr{y_p-\xip}{\mip}\in\R^2$ and so $|x_\infty|=\ell$. Thus we have obtained a contradiction which proves the assertion.
\end{proof}

\

\

\begin{proposition}\label{lemma:rapportoMuBounded} Let $(x_p)_p\subset\Omega$ be a family of points such that  $p |u_p(x_p)|^{p-1}\rightarrow +\infty$ and define $\mu_p$  in the usual way through $\mu_p^{-2}:=p |u_p(x_p)|^{p-1}$.
By $(\mathcal{P}_3^{k})$, up to a sequence, $R_{k,p}(x_p)=|x_{i,p}-x_p|$, for a certain  $i\in\{1,\dots,k\}$. Then, for $p$ large
\[
\frac{\mu_{i,p}}{\mu_p}\leq  1.
\]
\end{proposition}
\begin{proof}
Let us start by proving that $\frac{\mu_{i,p}}{\mu_p}$ is bounded for $p$ large. So we assume by contradiction that there exists a sequence $p_n\rightarrow +\infty$ such that
\begin{equation}\label{ipotAssurdo2}
\frac{\mu_{i,p_n}}{\mu_{p_n}}\rightarrow +\infty, \ \mbox{ as } p_n\rightarrow +\infty.
\end{equation}
By $(\mathcal{P}_3^{k})$ and \eqref{ipotAssurdo2} we have
\[
\frac{|x_{p_n}-x_{i,p_n}|}{\mu_{i, p_n}}=\frac{|x_{p_n}-x_{i,p_n}|}{\mu_{ p_n}}\frac{\mu_{p_n}}{\mu_{i,p_n}}\rightarrow 0,
\]
so that by $(\mathcal{P}_2^{k})$
\[
v_{i,p_n}\left(\frac{x_{p_n}-x_{i,p_n}}{\mu_{i, p_n}}\right)\rightarrow U(0)= 0.
\]
As a consequence
\[
\frac{\mu_{i,p_n}}{\mu_{p_n}}=\left(\frac{u_{p_n}(x_{p_n})}{u_{p_n}(x_{i,p_n})}\right)^{p_n-1}=\left(1+\frac{v_{i,p_n}\left(\frac{x_{p_n}-x_{i,p_n}}{\mu_{i,p_n}} \right)}{p_n}\right)^{p_n-1}\rightarrow e^{U(0)}=1,
\]
which contradicts \eqref{ipotAssurdo2}.

\

Next we show that $\tfrac\mip{\mu_p}\leq 1.$

 Assume by contradiction that there exists $\ell >1$ and a sequence $p_n\rightarrow +\infty$ such that
\begin{equation}\label{ipotAssurdo3}
\frac{\mu_{i,p_n}}{\mu_{p_n}}\rightarrow \ell, \ \mbox{ as } p_n\rightarrow +\infty.
\end{equation}
By $(\mathcal{P}_3^{k})$ and \eqref{ipotAssurdo3} we have
\[
\frac{|x_{p_n}-x_{i,p_n}|}{\mu_{i, p_n}}=\frac{|x_{p_n}-x_{i,p_n}|}{\mu_{ p_n}}\frac{\mu_{p_n}}{\mu_{i,p_n}}\leq \frac{2\sqrt C}{\ell}
\]
so that by $(\mathcal{P}_2^{k})$ there exists $x_\infty\in\mathbb R^2$, $|x_\infty|\leq \frac{2\sqrt C}{\ell}$ such that, up to a subsequence
\[
v_{i,p_n}\left(\frac{x_{p_n}-x_{i,p_n}}{\mu_{i, p_n}}\right)\rightarrow U(x_\infty)\leq 0.
\]
As a consequence
\[
\frac{\mu_{i,p_n}}{\mu_{p_n}}=\left(\frac{u_{p_n}(x_{p_n})}{u_{p_n}(x_{i,p_n})}\right)^{p_n-1}=\left(1+\frac{v_{i,p_n}\left(\frac{x_{p_n}-x_{i,p_n}}{\mu_{i,p_n}} \right)}{p_n}\right)^{p_n-1}\rightarrow e^{U(x_\infty)}.
\]
By \eqref{ipotAssurdo3} and the assumption $\ell >1$ we deduce
\[U(x_0)=\log \ell +o_p(1)>0\]
reaching a contradiction.
\end{proof}

\

\begin{proposition}\label{lemma:rappMuZero}
Let $x_p$ and $x_{i,p}$ be as in the statement of Proposition \ref{lemma:rapportoMuBounded}.
If \begin{equation}\label{ConditionNonVedo}\tfrac{|x_p-\xip|}{\mip}\to+\infty,\end{equation}
then \[\tfrac{\mip}{\mu_p}\to 0.\]
\end{proposition}
\begin{proof}
By Proposition \ref{lemma:rapportoMuBounded} we know that
\[
\frac{\mu_{i,p}}{\mu_{p}}\leq 1.
\]
Assume by contradiction that \eqref{ConditionNonVedo} holds but there exists $0<\ell \leq 1$ and a sequence $p_n\rightarrow +\infty$ such that
\begin{equation}\label{ipotAssurdo4}
\frac{\mu_{i,p_n}}{\mu_{p_n}}\rightarrow \ell, \ \mbox{ as } n\rightarrow +\infty.
\end{equation}
Then \eqref{ipotAssurdo4} and $(\mathcal{P}_3^{k})$ imply
\[
\frac{|x_{p_n}-x_{i,p_n}|}{\mu_{i, p_n}}=\frac{|x_{p_n}-x_{i,p_n}|}{\ell\  \mu_{p_n}}+o_{n}(1)\leq \frac{C}{\ell} +o_n(1) \ \mbox{ as }n\rightarrow +\infty
\]
which contradicts \eqref{ConditionNonVedo}.
\end{proof}

\

\begin{remark} \label{remarkSegnoNegativoMu}
If  $x_p$ and $x_{i,p}$ have opposite sign, i.e.
\[u_p(x_p)u_p(x_{i,p})<0,\]
then, by Corollary \ref{cor:nonvedoNL}, necessarily \eqref{ConditionNonVedo} holds.
Hence in this case \[\tfrac{\mip}{\mu_p}\to 0.\]
\end{remark}

\

\

\section{$G$-symmetric case: asymptotic analysis about the maximum points} \label{Section:GSymmetricMax}

In this section we start the asymptotic analysis which leads to the proof of Theorem \ref{TeoremaPrincipaleCasoSimmetrico}.

So we assume that $\Omega\subset\R^2$ is a $G$-symmetric domain as in the statement of Theorem \ref{TeoremaPrincipaleCasoSimmetrico}.

In particular we recall the hypothesis
\begin{equation}\label{primaHpSimm}
|G|\geq 4e.
\end{equation}
Then we consider a family $(\upp)$ of sign-changing $G$-symmetric solutions of \eqref{problem} with the properties listed in Theorem \ref{TeoremaPrincipaleCasoSimmetrico}.

In particular $u_p$ satisfies
\bel\label{assumptionEnergy}
p\int_{\Omega}|\nabla u_p|^2\leq \alpha\,8\pi e
\eel
for some $\alpha <5$ and $p$  large.

\

In the sequel we keep all notations introduced in Section \ref{section:intro} and Section \ref{SectionGeneralAnalysis} and add the following ones:
\begin{itemize}
\item $\mathcal{N}_p^{\pm}\subset\Omega$ denotes the positive/negative nodal domain of $u_p$

\item $\widetilde{\mathcal{N}}_p^{\pm}$ are the rescaled nodal domains about the points $x_p^{\pm}$ by the parameters $\mu_p^{\pm}$ defined in the introduction, i.e. \[\widetilde{\mathcal{N}}_p^{\pm}:=\frac{\mathcal{N}_p^{\pm}-x_p^{\pm}}{\mu_p^{\pm}}=\{x\in\mathbb R^2: x_p^{\pm}+\mu_p^{\pm}x\in \mathcal{N}_p^{\pm}\}.\]
\end{itemize}

\

We recall an energy lower bound (see for example \cite{DeMarchisIanniPacella}) and some obvious properties deriving from \eqref{assumptionEnergy}

\begin{lemma}\label{LemmaLowerUpper}
For any $\epsilon >0$ there exists $p_{\epsilon}$ such that
\begin{equation}\label{LowerEnergy}
pE_p(u_p^{\pm})\geq \  4\pi e -\ \epsilon \quad \forall p\geq p_{\epsilon}.
\end{equation}
Moreover
\[E_p(u_p)\rightarrow 0,  \ E_p(u_p^{\pm})\rightarrow 0, \ \|\nabla u_p\|_2\rightarrow 0, \ \|\nabla u_p^{\pm}\|_2\rightarrow 0.\]
\end{lemma}

\

From now on we assume w.l.o.g., as in Section \ref{section:intro}, that the $L^{\infty}$-norm of $u_p$ is achieved at the maximum point $x_p^+$ i.e. \[u_p(x_p^+)=\|u_p\|_{\infty}\geq -u_p(x_p^-).\]

Thanks to \eqref{assumptionEnergy} we are entitled to apply Proposition \ref{prop:x1N} to the solutions $(u_p)$.\\

\

\

For the rescaling about $x^+_p$ we have

\begin{proposition} \label{rem:x^+=x_1}
The rescaled function
\begin{equation}
v_p^+(x):=\fr{p}{\upp(x_p^+)}(\upp(x_p^++\mu_p^+ x)-\upp(x_p^+))
\end{equation}
defined on $\widetilde{\Omega}_{p}^+$ (see Section \ref{section:intro} for the definition) converges to $U$ in $C^1_{loc}(\mathbb R^2)$, where $U$ is the  function introduced in \eqref{v0}.

Moreover also the scaling of $u_p^+$ around $x_p^+$:
\begin{equation}
z_{p}^+(x):=z_{1,p}(x)=\fr{p}{\upp(x_p^+)}(\upp^+(x_p^++\mu_p^+ x)-\upp(x_p^+))
\end{equation}
defined on $\widetilde{\mathcal{N}}_{p}^+$ converges to $U$ in $C^1_{loc}(\mathbb R^2)$.
\end{proposition}

\begin{proof}
Since at $x_p^+$ the $L^{\infty}$-norm of $u_p$ is achieved, the proof of the convergence of $v_p^+$ is the same as that of \textit{Step 1} of Proposition \ref{prop:x1N}. The convergence of $z_p^+$ then comes from Corollary \ref{cor:nonvedoNL}.
\end{proof}

\

The previous Lemma \ref{LemmaLowerUpper} and Proposition \ref{rem:x^+=x_1} hold regardless the symmetry of $\Omega$. In the sequel using the symmetry assumptions on $\Omega$ and on our solution we will derive more specific and precise results.

\

Now we apply Proposition \ref{prop:x1N}, which gives a maximal number $k$ of families of points $(x_{i,p})$, $i=1,\ldots,k$, in $\Omega$ such that, up to a sequence, $(P^k_1)$, $(P^k_2)$ and $(P^k_3)$ hold for our solutions.

Then we get

\begin{proposition}\label{MaxVaazero} Defining $\mu_{i,p}$ as in \eqref{mip}, for $i=1,\dots,k$, we have
\[
\frac{|x_{i,p}|}{\mu_{i,p}}\ \mbox{ is bounded}
\]
and in particular $|x_{i,p}|\rightarrow 0$, $i=1,\dots,k$, namely the set of concentration points $\mathcal{S}=\{O\}$.
\end{proposition}
\begin{proof}

Without loss of generality we can assume that either $(x_{i,p})_p\subset \mathcal{N}_p^{+}$ or $(x_{i,p})_p\subset \mathcal{N}_p^{-}$. We prove the result in the case $(x_{i,p})_p\subset \mathcal{N}_p^{+}$, the other case being similar. Moreover in order to simplify the notation we drop the dependence on $i$ namely we set
$x_{p}:=x_{i,p}$ and $\mu_{p}:=\mu_{i,p}$.

Let $h:=|G|$ and let us denote by $g^j$, $j=0,\dots, h-1$, the elements of $G$. We consider the rescaled nodal domains
\[
\widetilde{\mathcal{N}_{p}}^{j,+} :=\{x\in\mathbb R^2\ : \ \mu_p x +g^jx_p\in \mathcal N_p^+\}
,\ \ j=0,\dots, h-1,\]
 and the rescaled functions $z_{p}^{j,+}(x): \widetilde{\mathcal{N}_{p}}^{j,+}\rightarrow\R$ defined by
\begin{equation}\label{z_j} z_{p}^{j,+}(x):=\frac{p}{u_{p}^+(x_{p})}\left( u_{p}^+(\mu_{p} x+g^jx_{p})-u_{p}^+(x_{p}) \right), \ \ j=0,\dots, h-1.\end{equation}
 Hence it's not difficult to see (as in Corollary \ref{cor:nonvedoNL}) that each $z_{p}^{j,+}$ converges to $U(x)=\log \left(\frac{1}{(1+\frac{1}{8}|x|^2)^2}\right)$ in $C^1_{loc}(\mathbb R^2)$, as $p\rightarrow \infty$ and $8\pi =\int_{\R^2}e^{U}dx$.\\

Assume by contradiction that there exists a sequence $p_n\rightarrow +\infty$ such that $\frac{|x_{p_n}|}{\mu_{p_n}}\rightarrow + \infty$.  Then, since the $h$ distinct points $g^j x_{p_n}$, $j=0,\ldots, h-1$, are the vertex of a regular polygon centered in $O$,   $d_n:=|g^j x_{p_n}-g^{j+1}x_{p_n}|=2\widetilde d_n \sin{\frac{\pi}{h}}$, where $\widetilde d_n:=|g^jx_{p_n}|$, $j=0,..,h-1$,  and so we also have that $\frac{d_n}{\mu_{p_n}}\rightarrow +\infty$.\\

Let \begin{equation}\label{R_n}R_{n}:=\min\left\{\frac{d_n}{3},\frac{d(x_{p_n},\partial\Omega)}{2},\frac{d(x_{p_n},NL_{p_n})}{2}\right\},
\end{equation}
then  by construction  $B_{R_n}(g^j x_{p_n})\subseteq \mathcal{N}_{p_n}^+$ for $j=0,\dots,h-1$,
\begin{equation}\label{palleDisgiunte} B_{R_n}(g^j x_{p_n})\cap B_{R_n}(g^l x_{p_n}) =\emptyset,\ \ \mbox{ for }j\neq l
\end{equation}
and
\begin{equation}\label{invadeR2}
\frac{R_n}{\mu_{p_n}}\rightarrow  +\infty.
\end{equation}

Using \eqref{invadeR2}, the convergence of $z_{p_n}^{j,+}$ to $U$,  \eqref{RemarkMaxCirca1} and Fatou's lemma, we have
\begin{eqnarray}\label{betterEstimate}
8\pi &=&\int_{\R^2}e^{U}dx\nonumber
\\
&\stackrel{\textrm{Fatou + conv.} + \eqref{invadeR2}}{\leq}& \lim_n \int_{ B_{\frac{R_n}{\mu_{p_n}}}(0)  } e^{z_{p_n}^{j,+} + (p_n+1)\left(\log{\left|1+\frac{z_{p_n}^{j,+}}{p_n}\right|}-\frac{z_{p_n}^{j,+}}{(p_n+1)}\right)}dx\nonumber
\\
&= &\lim_n \int_{B_{\frac{R_n}{\mu_{p_n}}}(0)}\left|1+\frac{z_{p_n}^{j,+}}{p_n} \right|^{(p_n+1)}dx\nonumber
\\
&=& \lim_n \int_{B_{\frac{R_n}{\mu_{p_n}}}(0)}\left|\frac{ u_{p_n}^+(\mu_{p_n} x+g^jx_{p_n})}{ u_{p_n}^+(x_{p_n})}dx     \right|^{(p_n+1)}dx\nonumber
\\
&=& \lim_n \int_{B_{R_n}(g^jx_{p_n})}\frac{\left| u_{p_n}^+\right|^{(p_n+1)}}{(\mu_{p_n})^2 \left|u_{p_n}^+(x_{p_n})\right|^{(p_n+1)}}dx\nonumber
\\
&=&\lim_n \frac{p_n}{\left|u_{p_n}^+(x_{p_n})\right|^2}  \int_{B_{R_n}(g^jx_{p_n})} \left| u_{p_n}^+\right|^{(p_n+1)}dx\nonumber
\\
&\stackrel{\eqref{RemarkMaxCirca1}}{\leq}&
 \lim_n p_n\int_{B_{R_n}(g^jx_{p_n})} \left| u_{p_n}^+\right|^{(p_n+1)}dx.
\end{eqnarray}
Summing on $j=0,\dots, h-1$, using \eqref{palleDisgiunte}, \eqref{assumptionEnergy}, \eqref{LowerEnergy} and \eqref{energiaSuSoluzioni} we get:
\begin{eqnarray*}
h\cdot 8\pi
&\leq &
\lim_n\ p_n \sum_{j=0}^{h-1} \int_{B_{R_n}(g^jx_{p_n})} \left| u_{p_n}^+\right|^{(p_n+1)}dx
\\
&\stackrel{\mbox{\eqref{palleDisgiunte}}}{\leq}&
\lim_n\   p_n\int_{\mathcal{N}_{p_n}^+} \left| u_{p_n}^+\right|^{(p_n+1)}dx
\\
&= &
 \lim_n\left(   p_n\int_{\Omega} \left| u_{p_n}\right|^{(p_n+1)}-\   p_n\int_{\mathcal{N}_{p_n}^-} \left| u_{p_n}^-\right|^{(p_n+1)}\right)dx
\\
&\stackrel{\eqref{assumptionEnergy} + \eqref{LowerEnergy}}\leq &
  \left(\alpha-1\right)\ \cdot 8\pi e
  \\
  &\stackrel{\alpha<5}< &
\  4 \ \cdot 8\pi e
  \end{eqnarray*}
  which is in contradiction with our assumption \eqref{primaHpSimm} on $|G|$.
  \end{proof}

\
\begin{remark}\label{rem:buonNormaInfinitoAbbassaSimmetria}
If we knew that $\|u_p\|_{\infty}\geq \sqrt{e}$, then we would obtain a better estimate in \eqref{betterEstimate}, and so Proposition \ref{MaxVaazero} would hold under the weaker symmetry assumption $|G|\geq 4$.
\end{remark}

\

\begin{corollary} \label{regioneNodaleInterna} We have:
\begin{itemize}
\item[$(i)$] $O\in \mathcal{N}_p^+$ for $p$ large.
\item[$(ii)$] $ x_{i,p}\in\mathcal{N}_p^+\mbox{ for } p\mbox{ large and  }i=1,\dots,k.$
\end{itemize}
\end{corollary}
\begin{proof}
By the properties of the solutions $(u_p)$ we know that the nodal line  $NL_p$ is the boundary of a domain containing $O$ in his interior. Hence  if $O\not\in\mathcal{N}_{p_n}^+$ for a sequence $p_n\rightarrow\infty$, it would follow that
\begin{equation}\label{distanza}
d(x_{p_n}^+,NL_{p_n})\leq |x_{p_n}^+|.
\end{equation}
Dividing by $\mu_{p_n}^+$ and passing to the limit, from Corollary \ref{cor:nonvedoNL} (remember that $x_{p_n}^+$ has the role of $x_{1,p_n}$ in the general Proposition \ref{prop:x1N}) we get that
\[\frac{|x_{p_n}^+|}{\mu_{p_n}^+}\rightarrow +\infty\]
which is a contradiction  with
  Proposition \ref{MaxVaazero}. So $(i)$ holds.

To prove $(ii)$ let us argue again by contradiction assuming that for a
sequence  $p_n\rightarrow +\infty$, $u_{p_n}(x_{i,p_n})<0$ holds for some $i\in\{1,\dots, k\}$. Then
\[
d(x_{i,p_n},NL_{p_n})\leq |x_{i,p_n}|
\]
so that, exactly with the same proof as in i), we reach a contradiction  with Proposition \ref{MaxVaazero}. So $(ii)$ holds.
\end{proof}

\

\begin{proposition}\label{prop:bark=1}
The maximal number  $k$ of families of points $(x_{i,p})$, $i=1,\ldots, k$, for which $(P^k_1)$, $(P^k_2)$ and  $(P^k_3)$ hold is $1$.
\end{proposition}
\begin{proof}
Let us assume by contradiction that $k > 1$ and set $x^+_p=x_{1,p}$. For a family $(x_{j,p})$, $j\in\{2,\ldots, k\}$ by  Proposition \ref{MaxVaazero}, there exists $C>0$ such that
\[
\frac{|x_{1,p}|}{\mu_{1,p}}\leq C\quad\textrm{and}\quad\frac{|x_{j,p}|}{\mu_{j,p}}\leq C.
\]
Thus, since by definition $\mu^+_p=\mu_{1,p}\leq \mu_{j,p}$, also
\[
\frac{|x_{1,p}|}{\mu_{j,p}}\leq C.
\]
Hence
\[
\frac{|x_{1,p}-x_{j,p}|}{\mu_{j,p}}\leq\frac{|x_{1,p}|+|x_{j,p}|}{\mu_{j,p}}\leq C\quad \textrm{as $p\to+\infty$},
\]
which  contradicts $(\mathcal{P}_1^k)$.
\end{proof}

\

\

Then we easily get

\begin{corollary}\label{prop:MinVaazero}
There exists $C>0$ such that
for any family $(x_p)_p\subset \Omega$, one has
\begin{equation}\label{boundDaQ1}
\frac{|x_p|}{\mu_p}\leq C
\end{equation}
for $p$ large.
\end{corollary}
\begin{proof}
By Proposition \ref{MaxVaazero}, \eqref{boundDaQ1}  holds for $x_{p}^+$.
Since, by Proposition \ref{prop:bark=1}, $k=1$, applying $(\mathcal{P}_3^1)$ to the points $(x_p)$, for $x_p\neq x_p^+$, we have
\[
\frac{|x_p-x_{p}^+|}{\mu_p}\leq C.
\]
By definition, $\mu_p^+\leq \mu_p$, hence we get
\[
\frac{|x_p|}{\mu_p}\leq \frac{|x_p-x_{p}^+|}{\mu_p}+\frac{|x_{p}^+|}{\mu_p}\leq\frac{|x_p-x_{p}^+|}{\mu_p}+\frac{|x_{p}^+|}{\mu_p^+}\leq C.
\]
\end{proof}

\begin{proposition}\label{prop:MaxVaazeroVelocemente} Let $(x_p)\subset\Omega$ be such that $\mu_p^{-2}:=p|\upp(x_p)|^{p-1}\to +\infty$ and assume that the rescaled functions $v_p(x):=\fr{p}{\upp(x_{p})}(\upp(x_{p}+\mu_{p} x)-\upp(x_{p}))$ converge to $U$ in $C^1_{loc}(\R^2\setminus\{-\lim_{p}\tfrac{x_p}{\mu_p}\})$ as $p\rightarrow +\infty$ ($U$ as in \eqref{v0}). Then
\begin{equation}\label{limiteZero}
\frac{|x_p|}{\mu_p}\to 0\ \ \mbox{ as }\ p\rightarrow +\infty.
\end{equation}
As a byproduct we deduce that $v_p\rightarrow U$ in $C^1_{loc}(\mathbb{R}^2\setminus\{0\})$, as $p\rightarrow +\infty$.
\end{proposition}
\begin{proof}
%
%

%
By Corollary \ref{prop:MinVaazero} we know that $\frac{|x_p|}{\mu_p}\leq C$. Assume by contradiction that $\fr{|x_p|}{\mu_p}\to\ell>0$. Let $g\in G$ such that  $|x_{p}-g x_{p}|=C_g|x_{p}|$ with constant $C_g> 1$ (such a $g$ exists because $G$ is a group of rotation about the origin).
Hence
\[
\frac{|x_{p}-g x_{p}|}{\mu_{p}}=C_g\frac{|x_{p}|}{\mu_{p}}\rightarrow C_g\ell > \ell.
\]
Then $x_{\infty}:=\lim_{p\rightarrow +\infty} \fr{g x_p-x_p}{\mu_p} \in\R^2\setminus\{-\lim_{p}\tfrac{x_p}{\mu_p}\}$ and so by the $C^1_{loc}$ convergence we get
\[
v_{p}(\fr{g x_p-x_p}{\mu_p})\to U(x_{\infty})<0\quad\textrm{as $p\to+\infty$}.\]
On the other side, for any $g\in G$, one also has
\[
v_{p}(\fr{gx_{p}-x_{p}}{\mu_{p}})=\fr{p}{\upp(x_{p})}(\upp(gx_{p})-\upp(x_{p}))=0,
\]
by the symmetry of $u_p$ and this gives a contradiction.
\end{proof}

\

\begin{proposition}\label{prop:distanzaxipNL}
Let $(x_p)$ be a family of points satisfying the same assumptions  as in Proposition \ref{prop:MaxVaazeroVelocemente}. Then
$$\mbox{either}\ \ \ \  \frac{dist (x_p, NL_p)}{\mu_p}\rightarrow +\infty\ \ \ \mbox{ or }\ \ \ \ \frac{dist (x_p, NL_p)}{\mu_p}\rightarrow 0 \ \ \ \ \mbox{ as $p\rightarrow +\infty$.}$$
 Moreover if $u_p(x_p)>0$ then $$\frac{dist (x_p, NL_p)}{\mu_p}\rightarrow +\infty \ \ \ \ \mbox{ as $p\rightarrow +\infty$.} $$
\end{proposition}
\begin{proof}
By Proposition \ref{prop:MaxVaazeroVelocemente} the rescaled functions $v_p$ converge to $U$ in $C^1_{loc}(\mathbb R^2\setminus \{0\})$. Therefore in order to prove the first assertion we can argue exactly as in the proof of Corollary \ref{cor:nonvedoNL} but now we cannot  exclude $\ell=0$ because we do not have the convergence of $v_p$ in the whole $\R^2$.\\

If instead we know that $u_p(x_p)>0$, then we show that the second alternative cannot occur.
Indeed, assume by contradiction that there exists $z_p\in NL_p$ such that $\frac{|x_p- z_p|}{\mu_p}\rightarrow 0$. Let $y_p\in\partial\Omega$ such that $\frac{|x_p- y_p|}{\mu_p}\rightarrow +\infty$ and define a continuous curve $\gamma_p:[0,1]\rightarrow \mathcal N_p^-$ such that $\gamma_p(0)=z_p$, $\gamma_p(1)=y_p$. Then, by continuity, there exists $t_p\in [0,1]$ such that   $\frac{|x_p- s_p|}{\mu_p}\rightarrow 1$ for $s_p:=\gamma_p(t_p)$. Therefore $v_p\left(\frac{s_p-x_p}{\mu_p} \right)\rightarrow U(x_{\infty})<0$ as $p\rightarrow +\infty$ for a point $x_{\infty}$ such that $|x_{\infty}|=1$. On the other hand, since $u_p(x_p)>0$, it follows also that
 $v_p\left(\frac{s_p-x_p}{\mu_p} \right)\leq -p \rightarrow -\infty$, giving a contradiction.
\end{proof}

\

\
%
%

\

\section{$G$-symmetric case: asymptotic analysis about the minimum points and proof of Theorem \ref{TeoremaPrincipaleCasoSimmetrico}}
\label{Section:GSymmetricMin}

As defined in the introduction we consider a family $(x_p^-)$ of minimum points of $u_p$.

By Lemma \ref{lemma:BoundEnergia} we have that $p|u_p(x_p^-)|^{p-1}\rightarrow +\infty$ as $p\rightarrow +\infty$. So defining $\mu_p^-$ by $(\mu_p^-)^{-2}:=p|u_p(x_p^-)|^{p-1} $, we have by $(\mathcal{P}_3^1)$ that
\begin{equation}\label{minimVaAZerotilde}
\frac{|x_p^+-x_p^-|}{\mu_p^-}\leq C
\end{equation}
for $p$ large. Moreover, since we already know that $\frac{d(x_p^+,NL_p)}{\mu_p^+}\rightarrow +\infty$ as $p\rightarrow +\infty,$ we deduce $\frac{|x_p^+-x_p^-|}{\mu_p^+}\rightarrow +\infty$ as $ p\rightarrow +\infty$, and in turn by \eqref{minimVaAZerotilde} we get
\begin{equation}
\label{rapportoMu}
\frac{\mu_p^+}{\mu_p^-}\rightarrow 0 \ \ \mbox{ as }\ p\rightarrow +\infty.
\end{equation}

Note that \eqref{minimVaAZerotilde} and \eqref{rapportoMu} more generally hold for any family of points $(x_p)$ such that $u_p(x_p)<0$ and $p|u_p(x_p)|^{p-1}\rightarrow +\infty$.

\

By Corollary \ref{prop:MinVaazero} we have
\bel\label{minimVaAZero}
 \frac{|x_p^-|}{\mu_p^-}\leq C,
\eel
so there are two possibilities: either $\fr{|x_p^-|}{\mu_p^-}\to\ell>0$ or $\fr{|x_p^-|}{\mu_p^-}\to0$ as $p\rightarrow +\infty$, up to subsequences. We will exclude the latter case.
\\

We start with a preliminary result:
\begin{lemma} \label{Lemma:scalingPreliminareOrigine}
For $x\in \frac{\Omega}{|x_p^-|}:=\{y\in\mathbb R^2\ :\ y |x_p^-|\in\Omega \}$ let us define the rescaled function
\[
w_p^-(x):=\frac{p}{u_p(x_p^-)}\left(u_p(|x_p^-|x) -u_p(x_p^-)\right).
\]
Then
\begin{equation}
\label{v1pv0versione2}
w_p^-\rightarrow \gamma\ \mbox{  in }\ C^1_{loc}(\R^2\setminus\{0\})\ \mbox{ as }\ p\rightarrow +\infty,
\end{equation}
 where $\gamma\in C^1(\R^2\setminus\{0\})$, $\gamma\leq0$, $\gamma(x_{\infty})=0$ for a point $x_{\infty}\in\partial B_1(0)$ and it is a solution to
 \[
 -\Delta \gamma =\ell^2 e^{\gamma} \ \mbox{ in }\ R^2\setminus\{0\}.
 \]
In particular $\gamma\equiv 0$ when  $\ell =0$.
\end{lemma}
\begin{proof}
\eqref{minimVaAZero} implies that $|x_p^-|\rightarrow 0$, so it follows that the set $\frac{\Omega}{|x_p^-|}\to\R^2$ as $p\to+\infty$.

By definition we have
\begin{equation}\label{conB}
w_p^-\leq 0\  \ \mbox{  and  }\ \  w_p(\frac{x_p^-}{|x_p^-|})=0
\end{equation}
and \[w_{p}^-=-p\quad\textrm{on $\partial \left( \frac{\Omega}{|x_p^-|} \right)$}.\]
Moreover, for $x\in \frac{\Omega}{|x_p^-|}$ we define $\xi_p:= |x_p^-|x$ and $\mu_{\xi_p}$ as $\mu_{\xi_p}^{-2}:= p|u_p(\xi_p)|^{p-1}$. Thanks to \eqref{problem} we then have
\begin{eqnarray}
\label{eq:RiscLineaNod}
|-\Delta w_p^-(x) |&=& \frac{p|x_p^-|^2|\upp(\xi_p)|^p}{|\upp(x^-_p)|}\nonumber
\\
&=& \frac{|\upp(\xi_p)|}{|\upp(x^-_p)|}\,\frac{|x^-_p|^2}{\mu_{\xi_p}^2}\nonumber
\\
&\leq& c_\infty \frac{|x^-_p|^2}{\mu_{\xi_p}^2},
\end{eqnarray}
where  $c_\infty:=\lim_p\|\upp\|_\infty$.
Then, observing that $\frac{|x_p^-|}{\mu_{\xi_p}}\leq \frac{C}{|x|}$  by  Corollary \ref{prop:MinVaazero} applied to $\xi_p$, we have
\be
|-\Delta w_p^-(x) |
\leq \frac{c_\infty C^2}{|x|^2}.
\ee

Namely for any $R>0$
\begin{equation}\label{boundlaplacianonew}|-\Delta w_p^-|\leq c_\infty C^2R^2\ \ \  \  \mbox{ in } \frac{\Omega}{|x_p^-|}\setminus B_{\frac{1}{R}}(0).\end{equation}

So, similarly as in \textit{Step 1} of the proof of Proposition  \ref{prop:x1N} (using now that  $w_p^-(\frac{x_p^-}{|x_p^-|})=0$), it follows that  for any $R>1$ ($\frac{x_p^-}{|x_p^-|}\in \partial B_1(0)\subset B_R(0)\setminus B_{\frac{1}{R}}(0)$ for $R>1$), $w_p^-$ is uniformly bounded in $B_R(0)\setminus B_{\frac{1}{R}}(0)$.

After passing to a subsequence, standard elliptic theory applied to the following equation
\bel\label{w-p}
-\lap w^-_p(x)=\frac{|x^-_p|^2}{(\mu^-_p)^2}\left(1+\frac{w^-_p(x)}{p}\right)\left|1+\frac{w^-_p(x)}{p}\right|^{p-1}
\eel
gives that $w_{p}^-$ is bounded in $ C^2_{loc}(\mathbb R^2\setminus\{0\})$ .
Hence \eqref{v1pv0versione2} and the properties of $\gamma$ follow.
\\

In particular when $\ell =0$ it follows that $\gamma$ is harmonic in $\R^2\setminus\{0\}$ and $\gamma(x_\infty)=0$ for some point $x_\infty\in\partial B_1(0)$, therefore by the maximum principle we obtain $\gamma\equiv0$.
\end{proof}

\

\

\begin{proposition}\label{prop:NLp+l>0}
There exists $\ell>0$ such that
\[
\fr{|x^-_p|}{\mu_p^-}\to\ell \ \ \ \mbox{ as }\ p\rightarrow +\infty.
\]
\end{proposition}

\begin{proof}
By  Corollary \ref{prop:MinVaazero} we know that $\fr{|x^-_p|}{\mu^-_p}\to\ell\in[0,+\infty)$. Let us suppose by contradiction that $\ell=0$.  Then  Lemma \ref{Lemma:scalingPreliminareOrigine} implies that \begin{equation}\label{wVaaZeroEqua}
w_p^-\rightarrow 0\ \mbox{ in }\ C^1_{loc}(\R^2\setminus\{0\}) \ \ \ \mbox{ as }\ p\rightarrow +\infty.
\end{equation}

\

By \eqref{problem}, applying the divergence theorem in $B_{|x^-_p|}(0)$ we get
\bel\label{ABCassurdo}
p\int_{\partial B_{|x^-_p|}(0)}\nabla \upp(y)\cdot\fr{y}{|y|}\,d\sigma(y)=p\int_{B_{|x^-_p|}(0)\cap\mathcal{N}^-_p}|\upp(x)|^p\,dx-p\int_{B_{|x^-_p|}(0)\cap\mathcal{N}^+_p}|\upp(x)|^p\,dx.
\eel
Scaling $\upp$ with respect to $|x^-_p|$ as in Lemma \ref{Lemma:scalingPreliminareOrigine}, by \eqref{wVaaZeroEqua} we obtain
\begin{eqnarray}\label{primo}
\left|p\int_{\partial B_{|x^-_p|}(0)}\nabla \upp(y)\cdot\fr{y}{|y|}\,d\sigma(y)\right|&=&\left|p\int_{\partial B_{1}(0)}|x^-_p|\nabla \upp(|x^-_p|x)\cdot\fr{x}{|x|}\,d\sigma(x)\right|\nonumber\\
&=&\left|\int_{\partial B_{1}(0)}\upp(x^-_p)\,\nabla w^-_p(x)\cdot\fr{x}{|x|}\,d\sigma(x)\right|\nonumber\\
&\leq& |\upp(x^-_p)|2\pi \sup_{|x|=1}|\nabla w^-_p(x)|=o_p(1).
\end{eqnarray}

Now we want to estimate the right hand side in \eqref{ABCassurdo}. We first observe that scaling around $|x^-_p|$ with respect to $\mu^-_p$ we get
\begin{eqnarray}\label{secondo}
p\int_{B_{|x^-_p|}(0)\cap\mathcal{N}^-_p}|\upp(x)|^p\,dx&=&p\int_{B_{1}(0)\cap\fr{\mathcal{N}^-_p}{|x^-_p|}}|\upp(|x^-_p|y)|^p|x^-_p|^2\,dy\nonumber\\
&\leq& c_\infty \int_{B_1(0)\cap\fr{\mathcal{N}^-_p}{|x^-_p|}}\fr{|\upp(|x^-_p|y)|^{p-1}}{|\upp(x^-_p)|^{p-1}}\fr{|x^-_p|^2}{(\mu_p^-)^2}dx\nonumber\\
&=&o_p(1),
\end{eqnarray}
where in the last equality we have used that $\fr{|\upp(|x^-_p|y)|^{p-1}}{|\upp(x^-_p)|^{p-1}}\leq1$, since $|x^-_p|y\in \mathcal{N}^-_p$ and that by assumption $\fr{|x^-_p|}{\mu_p^-}\to0$.

Next we claim that there exists $\bar p$ such that for any $p\geq \bar p$
\bel\label{inclusione}
B_{\mu^+_p}(x^+_p)\subset B_{|x^-_p|}(0).
\eel
Indeed, Corollary \ref{cor:nonvedoNL} implies that
\[+\infty = \lim_p \frac{d(x_p^+,NL_p)}{\mu_p^+}\leq \lim_p \frac{|x_p^+-x_p^-|}{\mu_p^+}\leq \lim_p\frac{|x_p^+|}{\mu_p^+}+\lim_p \frac{|x_p^-|}{\mu_p^+}=\lim_p\frac{|x_p^-|}{\mu_p^+},\]
where the last equality follows from Proposition \ref{prop:MaxVaazeroVelocemente} (i.e. $\frac{|x_p^+|}{\mu_p^+}\rightarrow 0$). Hence for any $x\in B_1(0)$ we have
\[
\fr{|x^+_p+\mu^+_p x|}{|x^-_p|}\leq\fr{|x^+_p|}{|x^-_p|}+\fr{\mu^+_p}{|x^-_p|}\leq\fr{2\mu^+_p}{|x^-_p|}\to0\ \ \mbox{ as }\ p\to+\infty,
\]
and so \eqref{inclusione} is proved.\\

Hence by \eqref{inclusione} and scaling around $x^+_p$ with respect to $\mu^+_p$ we obtain
\bel\label{terzo}
p\int_{B_{|x^-_p|}(0)\cap\mathcal{N}^+_p}|\upp(x)|^p\,dx\geq p\int_{B_{\mu^+_p}(x^+_p)}|\upp(x)|^{p}dx=c_\infty\int_{B_1(0)}e^{U}dx+o_p(1).
\eel
Collecting \eqref{ABCassurdo}, \eqref{primo}, \eqref{secondo} and \eqref{terzo} we get clearly a contradiction.

\end{proof}

\

\

Next we show that the nodal line shrinks to the origin faster than $\mu_p^-$ as $p\rightarrow +\infty$.

\begin{proposition}
\label{prop:NodalLineShrinks}
We have
\[
\frac{\max\limits_{y_p\in NL_p}|y_p|}{\mu_p^-}\rightarrow 0 \ \ \mbox{ as }\ p\rightarrow +\infty.
\]
\end{proposition}
\begin{proof} By Proposition \ref{prop:NLp+l>0} it is enough to prove that
\[
\frac{\max\limits_{y_p\in NL_p}|y_p|}{|x_p^-|}\rightarrow 0 \ \ \mbox{ as }\ p\rightarrow +\infty.
\]
First we show that, for any $y_p\in NL_p$, the following alternative holds:
\begin{equation}\label{AlternativaEq}
\mbox{either }\ \ \frac{|y_p|}{|x_p^-|}\rightarrow 0\ \ \mbox{ or }\ \ \ \frac{|y_p|}{|x_p^-|}\rightarrow +\infty\ \ \mbox{ as }p\rightarrow +\infty.
\end{equation}

\

Indeed assume by contradiction that $\frac{|y_p|}{|x_p^-|}\rightarrow m\in (0,+\infty)$.  Then $w_p^-(\frac{y_p}{|x_p^-|} )=-p\rightarrow -\infty$. But we have proved in Lemma \ref{Lemma:scalingPreliminareOrigine}  that $w_p^-(\frac{y_p}{|x_p^-|} )\rightarrow\gamma(y_{m})\in\R$, where $y_{m}$ is such that $|y_{m}|=m>0$, and this gives a contradiction.

To conclude the proof we have then to exclude the second alternative in \eqref{AlternativaEq}.
For $y_p\in NL_p$, let us assume by contradiction that $\frac{|y_p|}{|x_p^-|}\rightarrow +\infty$ and let us observe that
 \begin{equation}
 \label{unPuntoSullaLineaVaAZero}
 \exists\  z_p\in NL_p
 \ \mbox{  such that }\ \
\frac{|z_p|}{|x_p^-|}\rightarrow 0\ \mbox{ as }\ p\rightarrow +\infty.
\end{equation}
Indeed in the previous section we have shown that $O\in \mathcal N_p^+$,  hence there exists $t_p\in (0,1)$ such that $z_p:=t_p x_p^- \in NL_p$. Since $\frac{|z_p|}{|x_p^-|}<1$, by \eqref{AlternativaEq} we get \eqref{unPuntoSullaLineaVaAZero}.\\

Then for any $M>0$ there exists $\alpha_p^M\in NL_p$ such that $\frac{|\alpha_p^M|}{|x_p^-|}\rightarrow M$ and this is in contradiction with \eqref{AlternativaEq}.
\end{proof}

\

\

Finally we can analyze the local behavior of $u_p$ around the minimum point $x_p^-$. Note that by Proposition \ref{prop:MaxVaazeroVelocemente} and Proposition \ref{prop:NLp+l>0} we can already claim that the rescaling $v_p^-$ about $x_p^-$ (see \eqref{scalingNegativo} below) cannot converge to the regular solution $U$ of the Liouville problem \eqref{LiouvilleEquation}, such that $U(0)=0$ in $\R^2\setminus\{0\}$.

\begin{proposition}
\label{prop:scalingNegativo}
The scaling of $u_p$ around $x_p^-$
\begin{equation}
\label{scalingNegativo}
v_p^-(x):=\frac{p}{u_{p}(x_{p}^-)}\left( u_{p}(\mu_{p}^- x+x_{p}^-)-u_{p}(x_{p}^-) \right)
\end{equation}
defined on $\widetilde\Omega_p^-$ converges (passing to a subsequence)  in $C^1_{loc}(\mathbb R^2\setminus \{x_{\infty}\})$ to the function
\[V_{\ell}(x):=\log\left(\frac{2\alpha^2\beta^{\alpha}|x-x_{\infty}|^{\alpha -2}}{(\beta^{\alpha}+|x-x_{\infty}|^{\alpha})^2} \right),
\]
where $\alpha=\sqrt{2\ell^2+4}$, $\beta=\ell \left(\frac{\alpha+2}{\alpha-2} \right)^{1/\alpha}$ and $x_{\infty}\in\mathbb R^2$, $|x_{\infty}|=\ell$ and $\ell=\lim_p\fr{|x^-_p|}{\mu^-_p}>0$.

The function $V(x):=V_{\ell}(x+x_{\infty})$ is a radial singular solution of \eqref{LiouvilleSingularEquation} for $H=H(\ell)$.
\end{proposition}
\begin{proof}
Let us consider the translations of \eqref{scalingNegativo}:
\[
s_p^-(x):=v_p^-\left(x-\frac{x_p^-}{\mu_p^-} \right)=\fr{p}{\upp(x_{p}^-)}(\upp(\mu_p^- x)-\upp(x_p^-)),\quad \quad x\in \frac{\Omega}{\mu_p^-}
\]
which solve
\[
-\Delta s_p^-(x)=\left|1+\frac{s_p^-(x)}{p}    \right|^{p-1}\left(1+\frac{s_p^-(x)}{p}    \right),
\]
and
\[s^-_p(\fr{x^-_p}{\mu^-_p})=0,\]
\[s^-_p\leq 0.\]
Observe that $\frac{\Omega}{\mu_p^-}\rightarrow \mathbb R^2$ as $p\to+\infty$.

We claim that for any fixed $r>0$, $|-\Delta s_p^-|$ is bounded in $\frac{\Omega}{\mu_p^-}\setminus B_r(0)$. \\
Indeed Proposition \ref{prop:NodalLineShrinks} implies that if $x\in \frac{\mathcal{N}^+_p}{\mu_p^-}$,  then $|x|\leq\fr{\max\limits_{z_p\in NL_p}|z_p|}{\mu^-_p}<r,$ for $p$ large, hence
\[
\left( \frac{\Omega}{\mu_p^-}\setminus B_r(0) \right)
\subset\frac{\mathcal{N}^-_p}{\mu_p^-}\ \ \ \mbox{ for } p \mbox{ large}
\]
and so the claim follows observing that for $x\in \frac{\mathcal{N}^-_p}{\mu_p^-}$, then $|-\Delta s_p^-(x)|\leq1$.\\
Hence, by the arbitrariness of $r>0$, $s_p^-\rightarrow V$ in $C^1_{loc}(\mathbb R^2\setminus \{0\})$ where $V$ is a solution of
$$
-\Delta V=e^{V}\ \  \mbox{ in }\ \R^2\setminus\{0\}
$$
which satisfies $V\leq 0$ and $V(x_{\ell})=0$ where $x_{\ell}:=\lim_p\fr{x^-_p}{\mu^-_p}$ and $|x_{\ell}|=\ell$ by Proposition \ref{prop:NLp+l>0}.
Moreover $e^{V}\in L^1(\R^2)$, indeed for any $r>0$ and for any $\ep\in(0,1)$
\begin{eqnarray*}
\int_{B_{\fr1r}(0)\setminus B_r(0)}e^{V}\,dx&\leq&\int_{B_{\fr1r}(0)\setminus B_r(0)}\fr{|\upp(\mu^-_p x)|^{p+1}}{|\upp(x^-_p)|^{p+1}}dx+o_p(1)\\
&=&\fr{p}{|\upp(x^-_p)|^2}\int_{B_{\fr{\mu^-_p}r}(0)\setminus B_{r\mu^-_p}(0)}|\upp(y)|^{p+1}dy+o_p(1)\\
&\stackrel{Lemma\:\ref{lemma:BoundEnergia} (iii)}{\leq}&\fr{p}{(1-\ep)^2}\int_{\Omega}|\upp(y)|^{p+1}dy+o_p(1)\stackrel{Lemma\: \ref{LemmaLowerUpper}}{<}+\infty.
\end{eqnarray*}
Observe that if $V$ would solve $-\Delta V=e^V$ in the whole  $\R^2$ then necessarily $V(x)=U(x-x_{\ell})$. As a consequence $v_p^-(x)=s_p^-(x+\frac{x_p^-}{\mu_p^-})\rightarrow V(x+x_{\ell})=U(x)$ in $C^1_{loc}(\R^2\setminus\{-x_{\ell}\})$. Observe that $x_{\ell}=\lim_p \frac{x_p^-}{\mu_{p}^-}$ and so Proposition \ref{prop:MaxVaazeroVelocemente} would imply that $\frac{|x_p^-|}{\mu_p^-}\rightarrow 0$ as $p\rightarrow +\infty$, and this is in contradiction with Proposition \ref{prop:NLp+l>0}. Thus, by \cite{ChenLi2,ChouWan1,ChouWan2} and the classification in \cite{ChenLi} we have that $V$ solves, for some $\eta >0$, the following entire equation
\bel
\left\{
  \begin{array}{ll}
    -\lap V=e^{V}-4\pi\eta\delta_0 & \hbox{\textrm{in $\R^2$}} \\
    \int_{\R^2}e^{V}dx=8\pi(1+\eta), & \hbox{\:}
  \end{array}
\right.
\eel
where $\delta_0$ denotes the Dirac measure centered at the origin.

\

We claim that $V$ is radial.

Indeed, by the classification given in \cite{PrajapatTarantello}, we have that either $V$ is radial, or $\eta\in\N$ and $V$ is $(\eta+1)$-symmetric.
Suppose by contradiction that $\eta\in\N$ and $V$ is $(\eta+1)$-symmetric, then, since $V$ is the limit of $s^-_p$ (which is $G$-symmetric with $|G|\geq 4 \, e$) we get $\eta+1\geq 4\,e$ and so
\bel\label{maggiore}\int_{\R^2}e^{V}dx\geq 4\,e\cdot8\pi.
\eel
On the other hand for any $R>0$:
\begin{eqnarray}\label{minore}
\int_{B_R(0)\setminus B_{\fr1R}(0)}e^{V}dx&\leq & \lim_{p\rightarrow +\infty}\int_{B_R(0)\setminus B_{\fr1R}(0)}\left|\fr{\upp(\mu^-_p x)}{\upp(x^-_p)}\right|^{p+1}dx\nonumber\\
&=& \lim_{p\rightarrow +\infty} \fr p{\upp^2(x^-_p)}\int_{B_{R\mu^-_p}(0)\setminus B_{\fr{\mu^-_p}R}(0)}|\upp(y)|^{p+1}dy\nonumber\\
&\stackrel{(*)}\leq& \lim_{p\rightarrow +\infty} \fr p{\upp^2(x^-_p)}\left(\int_\Omega |\upp(y)|^{p+1}dy-\int_{\mathcal{N}^+_p} |\upp(y)|^{p+1}dy\right)\nonumber\\
&\stackrel{(\sharp)}\leq&(\alpha-1)\cdot 8\pi e
\end{eqnarray}
where in $(*)$ we have used that, by Proposition \ref{prop:NodalLineShrinks}, $\mathcal{N}^+_p\subset B_{\fr{\mu^-_p}R}(0)$ and   in $(\sharp)$ we have applied \eqref{RemarkMaxCirca1}, Lemma \ref{LemmaLowerUpper} and \eqref{assumptionEnergy}.
By the arbitrariness of $R$  from \eqref{minore}
we then get
\begin{equation}\label{minoreNew}
\int_{\R^2}e^{V}dx\leq (\alpha-1)\;e\cdot  8\pi.
\end{equation}

Last, using the assumption $\alpha< 5$ in \eqref{minoreNew} we get a contradiction with \eqref{maggiore}.

\

Thus $V$ is radial and $V(r)$ satisfies
\[
\left\{
\begin{array}{lr}-V''-\frac{1}{r}V'=e^{V}\  \mbox{ in } (0, +\infty)\\
V\leq 0\\
V(\ell)=V'(\ell)=0
\end{array}
\right..
\]
The solutions of this problem are
\begin{equation}\label{soluzioneGenerale}
V(r)=\log\left(\frac{4}{\delta^2}\frac{e^{\frac{\sqrt{2}}{\delta}(\log r-y))}}{\left( 1+e^{\frac{\sqrt{2}}{\delta}(\log r-y))}\right)^2}   \right)
\end{equation}
for $\delta>0, y\in\mathbb R$.

Observe that from $V'(r)=0$ we get $\frac{1-\sqrt{2}\delta}{1+\sqrt{2}\delta}=e^{\frac{\sqrt{2}}{\delta}(\log r-y)}$ and moreover $V(r)=0$ for $r=\frac{\sqrt{1-2\delta^2}}{\delta}$.
Hence by $V(\ell)=V'(\ell)=0$ it follows that $\ell^2=\frac{1-2\delta^2}{\delta^2}$ which implies that $\delta=\frac{1}{\sqrt{2+\ell^2}}$.
Inserting this estimate into \eqref{soluzioneGenerale} we get
\[
V(r)=\log\left(\frac{2\alpha^2\beta^{\alpha}r^{\alpha -2}}{(\beta^{\alpha}+r^{\alpha})^2} \right),
\]
where $\alpha=\sqrt{2\ell^2+4}$ and $\beta=\ell \left(\frac{\alpha+2}{\alpha-2} \right)^{1/\alpha}$.\\

The conclusion follows observing that $v_p^-(x)=s_p^-\left( x+\frac{x_p^-}{\mu_p^-}\right)$.
\end{proof}

\

\

\

\begin{proof}[Proof of Theorem \ref{TeoremaPrincipaleCasoSimmetrico}. ]
 It follows from all previous results. More precisely, i) follows from \eqref{boundDaQ1} and Lemma \ref{lemma:BoundEnergia}. The statement ii) is from Proposition \ref{prop:NodalLineShrinks}. Finally the asymptotic behavior of the rescaled functions $v_p^+$ and $v_p^-$ are shown in Proposition \ref{rem:x^+=x_1} and Proposition \ref{prop:scalingNegativo}.
\end{proof}

\

\

\begin{remark}
By \eqref{rapportoMu} applied to any $(x_p)_p\subset\Omega$
such that $u_p(x_p)<0$, $\mu_p^{-2}=p|u_p(x_p)|^{p-1}\to +\infty$, we easily derive
 \[p\left(u_p(x_p^+)+u_p(x_p)\right)\rightarrow + \infty.\]

Indeed if by contradiction
\begin{equation} \label{contradiction} p_n\left(u_{p_n}(x_{p_n}^+)+u_{p_n}(x_{p_n})\right)\rightarrow K\geq 0\end{equation}
Then, setting $c_{\infty}:=\lim_p u_p(x_p^+)>0$, we would have
\[
\frac{{\mu^+_p}^2}{{\mu_p}^2}=\left(\fr{|\upp(x_p)|}{\upp(x_p^+)}\right)^{p-1}=\left(1-\frac{\frac{p(\upp(x^+_p)+\upp(x_p))}{\upp(x^+_p)}}{p}\right)^{p-1}\stackrel{p\to+\infty}{\to} e^{-K/c_{\infty}}\neq 0,
\]
which is in contradiction with \eqref{rapportoMu}.

Thus in particular
\[p\left(u_p(x_p^+)+u_p(x_p^-)\right)\rightarrow + \infty,\]
which means, in the notation of \cite{GrossiGrumiauPacella1}, that $\upp$ is of type $B'$.
\end{remark}

\

\begin{remark}
It is not difficult to prove an analogous of Theorem \ref{TeoremaPrincipaleCasoSimmetrico} for higher energy solutions, under stronger symmetry assumptions. Precisely one could substitute the assumptions
\eqref{primaHpSimm} and \eqref{assumptionEnergy} by
\begin{equation}\label{primaHpSimmmmmm}
|G|\geq me
\end{equation}
and
\bel\label{assumptionEnergymmmmmmm}
p\int_{\Omega}|\nabla u_p|^2\leq \alpha\,8\pi e
\eel
for some $\alpha < m+1$ and $p$  large, for any $m\in\N\setminus\{0\}$.
\end{remark}

\

\

\section{Further results and open questions}\label{sectionOpenProblems}
The asymptotic result of Theorem \ref{TeoremaPrincipaleCasoSimmetrico} together with the existence result of  \cite{DeMarchisIanniPacella} shows the presence of sign changing $G$-symmetric solutions of  \eqref{problem} whose limit profile, as $p\rightarrow +\infty$, looks like the superposition of (at least) two different signed bubbles coming, roughly speaking, from a regular and a singular solution of  \eqref{LiouvilleEquation} and  \eqref{LiouvilleSingularEquation}.

\

The two bumps could carry different  energies but we cannot precisely estimate them so to say that they ``exhaust" all the energy of the solutions $u_p$ which is bounded by \eqref{assumptionEnergyINIZIO}. This means that ``a priori" one could think that other bumps could develop as $p\rightarrow +\infty$. We believe that this is not the case, as it happens in the radial setting studied in \cite{GrossiGrumiauPacella2}.

\

A partial result in this direction is obtained in the next proposition which exclude the presence of other positive bumps having the limit profile of a regular solution of \eqref{LiouvilleEquation}.

\

\begin{proposition}\label{prop:N=1} Under the same assumptions of Theorem \ref{TeoremaPrincipaleCasoSimmetrico},
 let $(x_p)\subset\Omega$ be such that $\mu_p^{-2}:=p|\upp(x_p)|^{p-1}\to +\infty$ and assume that
\[u_p(x_p)>0\]
and that the rescaled functions $v_p(x):=\fr{p}{\upp(x_{p})}(\upp(x_{p}+\mu_{p} x)-\upp(x_{p}))$ converge to $U$ in
$C^1_{loc}(\R^2\setminus\{0\})$. Then
\[x_p=x^+_p+o_p(1)\mu^+_p\]
\[\frac{\mu^+_p}{\mu_p}\to 1\ \ \ \mbox{ as } p\rightarrow +\infty\]
\[u_p(x_p)\rightarrow c_{\infty}\ \ \ \mbox{ as } p\rightarrow +\infty\]
where $c_{\infty}:=\lim_p \|u_p^+\|_{\infty}$.\\
So, roughly speaking, scaling about $x_p$ with respect to its parameter we obtain the same bubble appearing from the scaling about $x^+_p$ with respect to $\mu_p^+$.
\end{proposition}
\begin{proof}
\textit{Step 1. The following alternative holds:
\begin{equation}\label{alternativaPositive}
\mbox{either }\  \  \ \frac{|x^+_p-x_p|}{\mu_p^+}\to 0\quad\textrm{or}\quad \frac{|x^+_p-x_p|}{\mu_p^+}\to+\infty\ \ \ \mbox{ as } p\rightarrow +\infty.
\end{equation}}

Indeed if by contradiction there exists $C>0$ such that $\frac{|x^+_p-x_p|}{\mu_p^+}\to C$, then by Proposition \ref{rem:x^+=x_1} we get, for  $x_C:=\lim_p \frac{x_p-x_p^+}{\mu_p^+}$, $x_C\in\partial B_C(0)$:
\[
v^+_p\left(\frac{x_p-x^+_p}{\mu^+_p}\right)\to U(x_C)\in (-\infty, 0) \ \ \ \mbox{ as } p\rightarrow +\infty
\]
and so
\[
\frac{|x^+_p-x_p|}{\mu_p}=
%
%
=\frac{|x^+_p-x_p|}{\mu_p^+}\left(1+\frac{v_p^+(\frac{x_p-x_p^+}{\mu_p^+})}p\right)^{\frac{p-1}2}\to C e^{\frac{U(x_\infty)}2}>0 \ \ \ \mbox{ as } p\rightarrow +\infty.
\]
This yields to a contradiction because by Proposition \ref{prop:MaxVaazeroVelocemente}
\[
\frac{|x^+_p-x_p|}{\mu_p}\leq\frac{|x_p|}{\mu_p}+\frac{|x^+_p|}{\mu^+_p}\to0\ \ \ \mbox{ as } p\rightarrow +\infty.
\]

\

\

\textit{Step 2. We prove that only the first alternative in \eqref{alternativaPositive} holds, namely
\begin{equation}\label{onlyTheFirstAlternative}
 \frac{|x^+_p-x_p|}{\mu_p^+}\to 0 \ \ \ \mbox{ as }\ p\rightarrow +\infty.
\end{equation} }

Suppose by contradiction that $\tfrac{|x^+_p-x_p|}{\mu_p^+}\to+\infty$.
As a consequence, by Proposition \ref{lemma:rappMuZero}, we have \bel\label{rapporto}\tfrac{\mu^+_p}{\mu_p}\to0.\eel

By the divergence theorem, for any $r>0$ and $p\geq p_r$ we also have:
\begin{eqnarray}\label{ABCp}
-p\int_{\partial B_{r\mup}(\xp)}\nabla \upp(y)\cdot\frac{y-x_{p}}{|y-x_{p}|}\,d\sigma(y)&=&-p\int_{B_{r\mup}(\xp)}\lap\upp(x)\,dx=\nonumber\\
&{=}& p\int_{B_{r\mup}(\xp)}|\upp(x)|^p\,dx,
\end{eqnarray}
where for  the last equality we have used \eqref{problem}, the assumption $u_p(x_p)>0$ and Proposition \ref{prop:distanzaxipNL} to deduce that, for $p\geq p_r$, $B_{r\mu_p}(x_p)\subset{\mathcal N}_p^+$.\\

Now, since the function $U$ introduced in \eqref{v0} is in $C^\infty(\R^2)$, it is possible to fix $r>0$ fulfilling the following condition:
\begin{equation}\label{r12}
2\pi r \sup_{|x|= r}|\nabla U(x)|\leq\fr23\int_{B_1(0)}e^{U}\,dx.
\end{equation}
With this choice of $r$ we estimate the two terms of \eqref{ABCp}.\\

By Proposition \ref{prop:MaxVaazeroVelocemente} and \eqref{rapporto}, there exists $p_r'$ such that for any $p\geq p'_r$ $B_{\mu^+_p}(x^+_p)\subset B_{r\mup}(\xp)$; moreover, using the convergence of $v_p^+$ to $U$ in $C^1_{loc}(\mathbb R^2)$, we get
\begin{eqnarray}\label{Bp}
p\int_{B_{r\mup}(\xp)}|\upp(x)|^p\,dx&\geq&  p\int_{B_{\mu_p^+}(x_p^+)}|\upp(x)|^p\,dx \nonumber\\
&=& {\int_{B_1(0)}\fr{|\upp(x_p^++\mu_p^+ y)|^{p}}{|\upp(x_p^+)|^{p-1}}\,dy}\nonumber\\
 &=&{u_p(x_p^+)\int_{B_1(0)}\left|1+\fr{\upp(x_p^++\mu_p^+ y)-\upp(x_p^+)}{\upp(x_p^+)}\right|^{p}\,dy}\nonumber\\
  &=&{u_p(x_p^+)\int_{B_1(0)}\left|1+\fr{v_p^+(y)}{p}\right|^{p}\,dy}\nonumber\\
&=& c_{\infty}\int_{B_1(0)}e^{U}+o_p(1),
\end{eqnarray}
where $c_{\infty}:=\lim_{p\rightarrow +\infty}\|u_p\|_{\infty}$.\\

Finally, rescaling $\upp$ around $\xp$ with respect to $\mup$, by the convergence of $v_{p}$ to $U$ in $C^1_{loc}(\R^2\setminus\{0\})$ we obtain for $p\geq p''_r$
\begin{eqnarray}\label{Cp}
\left|p\int_{\partial B_{r\mup}(\xp)}\nabla \upp(y)\cdot\frac{y-x_{p}}{|y-x_{p}|}\,d\sigma(y)\right|
&=&
\left|\int_{\partial B_r(0)}\upp(\xp)\nabla v_{p}(x)\cdot\fr{ x}{| x|}\,d\sigma(x)\right|\nonumber
\\
&=&
\upp(\xp)\,\left|\int_{\partial B_r(0)}\nabla v_{p}(x)\cdot\fr{ x}{| x|}\,d\sigma(x)\right|\nonumber
\\
&\leq&
\upp(\xp)\,2\pi r \sup_{|x|= r}\ |\nabla v_{p}(x)|\nonumber
\\
&\leq & c_\infty 2\pi r \sup_{|x|= r}|\nabla U(x)|+o_p(1).
\end{eqnarray}

In conclusion, by our choice of $r$, collecting \eqref{Bp} and \eqref{Cp} we derive for $p\geq\max\{p_r,p'_r,p''_r\}$:
\begin{eqnarray*}
0<c_\infty\int_{B_1(0)}e^{U}\,dx+o_p(1)&\leq& p\int_{B_{r\mup}(\xp)}|\upp(x)|^p\,dx\nonumber\\
&\!\!\!\!\!\!\!\!\!\!\!\!\!\!\!\!\!\!\!\!\!\!\!\!\!\!\!\!\!\!=&\!\!\!\!\!\!\!\!\!\!\!\!\!\!\!\!\!\!\left|p\int_{\partial B_{r\mup}(\xp)}\nabla \upp(y)\cdot\frac{y-\xp}{|y-\xp|}\,d\sigma(y)\right|\nonumber\\
&\!\!\!\!\!\!\!\!\!\!\!\!\!\!\!\!\!\!\!\!\!\!\!\!\!\!\!\!\!\!\leq&\!\!\!\!\!\!\!\!\!\!\!\!\!\!\!\!\!\!c_\infty 2\pi r \sup_{|x|= r}|\nabla U(x)|+o_p(1)\stackrel{\eqref{r12}}{\leq} c_\infty \fr23\int_{B_1(0)}e^{U}\,dx+o_p(1),
\end{eqnarray*}
which is clearly a contradiction.

\

\textit{Step 3. Conclusion of the proof.}

By \eqref{onlyTheFirstAlternative} and Proposition \ref{rem:x^+=x_1} we get:
\[
v^+_p\left(\frac{x_p-x^+_p}{\mu^+_p}\right)\to U(0)=0 \ \ \ \mbox{ as } p\rightarrow +\infty,
\]
and so
\[
\left(\frac{\mu^+_p}{\mu_p}\right)^2=\left(\frac{\upp(x_p)}{\upp(x^+_p)}\right)^{p-1}=\left(1+\frac{v^+_p\left(\frac{x_p-x^+_p}{\mu^+_p}\right)}p\right)^{p-1}\to1 \ \ \ \mbox{ as } p\rightarrow +\infty
\]
and
\[
\frac{u_p(x_p)}{u_p(x_p^+)}-1=\frac{1}{p}\ v^+_p\left(\frac{x_p-x^+_p}{\mu^+_p}\right) \rightarrow 0 \ \ \ \mbox{ as } p\rightarrow +\infty.
\]
\end{proof}

\

\

\begin{remark} We are not able to get a similar  result in the negative nodal region i.e. for points $(x_p)\subset\Omega$ such that $u_p(x_p)<0$ and  $\mu_p^{-2}:=p|\upp(x_p)|^{p-1}\to +\infty$.

In this case, using  Proposition \ref{prop:NodalLineShrinks}, Corollary \ref{prop:MinVaazero} and Proposition \ref{prop:NLp+l>0} it is easy to get that
\begin{equation}\label{pa}
\max_{y_p\in NL_p}\frac{|x_p-y_p|}{\mu_p}\leq C
\ \mbox{ and  }\
\frac{|x_p-x_p^-|}{\mu_p}\leq C
\end{equation}
for $p$ large, which seems to indicate that there are no other negative bumps other than the one previously found.
\end{remark}

As previously said the main reason why we cannot exclude the presence of other bubbles, under the hypothesis of Theorem \ref{TeoremaPrincipaleCasoSimmetrico}, is that we cannot precisely estimate the energy carried by each bubble so to use the bound \eqref{assumptionEnergyINIZIO}  to say that the two bubbles given by rescaling about $x_p^+$ and $x_p^-$ use all the available energy. Let us point out that the energy carried by each of these bubbles depends on two quantities:
\begin{itemize}
\item[$i)$] the energy of the solution of the limit problem (related to the bubble)
\item[$ii)$] the limit values of $u_p(x_p^+)$ or $u_p(x_p^-)$.
\end{itemize}
In the case of the positive bubble obtained by rescaling about $x_p^+$ we know $i)$ but we miss a good estimate of $u_p(x_p^+)$ in $ii)$. Motivated by the results concerning the radial situation (\cite{GrossiGrumiauPacella2}) we conjecture that

$$
\mbox{\bf{(C1)}}\qquad\qquad \lim_p u_p(x_p^+)=A^+>\sqrt{e}.\qquad\qquad\qquad\qquad\qquad\qquad\quad\;
$$

Note that if we knew this we could reduce the assumption on the symmetry group $G$, by just requiring $|G|\geq 4$, as in \cite{DeMarchisIanniPacella} (see the proof of Proposition \ref{MaxVaazero} and Remark \ref{rem:buonNormaInfinitoAbbassaSimmetria}).

In the case of the negative bubble obtained by rescaling about $x_p^-$ we neither have a good estimate of the energy of the singular solution of the limit problem (since it depends on the constant $\ell=\lim_{p\rightarrow +\infty}\frac{|x_p^-|}{\mu_p^-}>0$) nor a good estimate of $u_p(x_p^-)$.
Thinking again of the radial solution (\cite{GrossiGrumiauPacella2}) we conjecture that
$$
\mbox{\bf{(C2)}}\qquad\qquad 
\lim_p p\int_{\Omega} |\nabla u_p^-|=B^-> 8\pi e \qquad\qquad\qquad\qquad\qquad\qquad
$$
and 
$$
\mbox{\bf{(C3)}}\qquad\qquad 
\lim_p u_p(x_p^-)=A^-,\  \ 1<A^-<\sqrt{e}.\qquad\qquad\qquad\qquad\quad\;
$$
More generally we believe that estimates analogue to  {\bf (C1)}, {\bf (C2)} and {\bf (C3)} should hold for bubble tower solutions of \eqref{problem} in general domains.


\begin{thebibliography}{99}
\bibitem{AdiGrossi} Adimurthi, M. Grossi, \emph{Asymptotic estimates  for a two-dimensional problem with polynomial nonlinearity}, Proc. Amer. Math. Soc. 132 (2004), no. 4, 1013-1019.

\bibitem{AdiStruwe}  Adimurthi, M. Struwe, \emph{Global compactness properties of semilinear elliptic equations with critical exponential growth}, J. Funct. Anal. 175 (2000), no. 1, 125-167.

\bibitem{BenAyed_ElMehdi_Pacella} M. Ben Ayed, K. El Mehdi, F. Pacella, \emph{Classification of low energy sign-changing solutions of an almost critical problem}, J. Funct. Anal. 250 (2007), no. 2, 347-373.

\bibitem{Brezis} H. Brezis, \emph{Probl\`emes de convergence dans certaines EDP non lin\'eaires et applications g\'eom\'etriques}, Goulaouic-Meyer-Schwartz seminar, 1983–1984, Exp. No. 14, 11 pp., \'Ecole Polytech., Palaiseau, 1984.


\bibitem{BC1} H. Brezis, J.-M. Coron, \emph{Convergence de solutions de H-syst\`emes et application aux surfaces \`a courbure moyenne constante}, C. R. Acad. Sci. Paris S\'er. I Math. 298 (1984), no. 16, 389-392.

\bibitem{BC2}
H. Brezis, J.-M. Coron, \emph{Convergence of solutions of H-systems or how to blow bubbles}, Arch. Rational Mech. Anal. 89 (1985), no. 1, 21-56.

\bibitem{CazenaveDicksteinWeissler} T. Cazenave, F. Dickstein, F.B. Weissler, \emph{Sign-changing stationary solutions and blowup for the nonlinear heat equation in a ball}, Math. Ann. 344 (2009), no. 2, 431-449.

\bibitem{ChenLi} W.X. Chen, C. Li, \emph{Classification of solutions of some nonlinear elliptic equations}, Duke Math. J. 63 (1991), 615-622.

\bibitem{ChenLi2} W.X. Chen, C. Li, \emph{Qualitative properties of solutions of some nonlinear elliptic equations}, Duke Math. J. 71 (1993), 427-439.

\bibitem{ChouWan1} K.S. Chou, T. Wan, \emph{Asymptotic radial symmetry for solutions of $\Delta u+e^u=0$ in a punctured disc}, Pacific J. Math. 163 (1994), no. 2, 269-276.

\bibitem{ChouWan2} K.S. Chou, T. Wan, \emph{Correction to: "Asymptotic radial symmetry for solutions of $\Delta u+e^u=0$ in a punctured disc'' [Pacific J. Math. 163 (1994), no. 2, 269-276]}, Pacific J. Math. 171 (1995), no. 2, 589-590.

\bibitem{DeMarchisIanniPacella} F. De Marchis, I. Ianni, F. Pacella, \emph{Sign changing solutions of Lane Emden problems with interior nodal line and semilinear heat equations}, Journal of Differential Equations 254 (2013), 3596-3614.

\bibitem{DicksteinPacellaSciunzi} F. Dickstein, F. Pacella, B. Sciunzi, \emph{Sign-changing stationary solutions and blowup for the nonlinear heat equation in dimension two}, (2013), arXiv:1304.2571.

\bibitem{Druet} O. Druet, \emph{Multibumps analysis in dimension 2: quantification of blow-up levels}, Duke Math. J. 132 (2006), no. 2, 217-269.

\bibitem{DruetHebeyRobert} O. Druet, E. Hebey, F. Robert, \emph{Blow-up theory for elliptic PDEs in Riemannian geometry}, Mathematical Notes, 45, Princeton University Press, Princeton, NJ, 2004. ISBN: 0-691-11953-8

\bibitem{GrossiGrumiauPacella1} M. Grossi, C. Grumiau, F. Pacella, \emph{Lane Emden problems: asymptotic behavior of low energy nodal solutions}, Ann. Inst. H. Poincar\'{e} Anal. Non Lin\'{e}aire 30 (2013), no. 1, 121-140.

\bibitem{GrossiGrumiauPacella2}
M. Grossi, C. Grumiau, F. Pacella, \emph{Lane Emden problems with large exponents and singular Liouville equations}, (2012) to appear in J. Math Pures Appl.. arXiv:1209.1534.

\bibitem{GrossiPacella} M. Grossi, F. Pacella, \emph{Positive solutions of nonlinear elliptic equations with critical Sobolev exponent and mixed boundary conditions}, Proc. Roy. Soc. Edinburgh Sect. A 116 (1990), no. 1-2, 23-43.

\bibitem{GrossiPistoia}
M. Grossi, A. Pistoia, \emph{Multiple blow-up phenomena for the sinh-Poisson equation}, Arch. Rat. Mech. Anal. 209 (2013), no. 1, 287-320.

\bibitem{MarinoPacellaSciunzi}
V. Marino, F. Pacella, B. Sciunzi, \emph{Blow up of solutions of semilinear heat equations in general domains}, (2013) to appear in Comm. Cont. Math.. arXiv:1306.1417.

\bibitem{PrajapatTarantello}
J. Prajapat, G. Tarantello, \emph{On a class of elliptic problems in $\R^2$: Symmetry and Uniqueness results}, Proc. Roy. Soc. Edinburgh 131A (2001), 967-985.

\bibitem{RenWei1} X. Ren, Xiaofeng, J. Wei, \emph{On a two-dimensional elliptic problem with large exponent in nonlinearity}, Trans. Amer. Math. Soc. 343 (1994), no. 2, 749-763.

\bibitem{RenWei2}
X. Ren, Xiaofeng, J. Wei, \emph{Single-point condensation and least-energy solutions}, Proc. Amer. Math. Soc. 124 (1996), no. 1, 111-120.

\bibitem{SantraWei}
S. Santra, J. Wei,  \emph{Asymptotic behavior of solutions of a biharmonic Dirichlet problem with large exponents}, J. Anal. Math. 115 (2011), 1-31.

\bibitem{Struwe_Book} M. Struwe, \emph{Variational methods. Applications to nonlinear partial differential equations and Hamiltonian systems}, Springer-Verlag, Berlin, 1990. ISBN: 3-540-52022-8

\end{thebibliography}
\end{document}